\documentclass[12pt]{amsart}
\usepackage{MnSymbol}
\usepackage{mathrsfs}
\usepackage[utf8]{inputenc}
\usepackage{amsmath,amsthm}

\usepackage{tikz-cd}

\newtheorem{Theorem}{Theorem}[section]
\newtheorem{Corollary}[Theorem]{Corollary}
\newtheorem{Lemma}[Theorem]{Lemma}
\newtheorem{Proposition}[Theorem]{Proposition}
\theoremstyle{definition}
\newtheorem{Remark}[Theorem]{Remark}

\numberwithin{equation}{section}

\DeclareMathAlphabet\mathbb{U}{msb}{m}{n}
\newcommand{\mono}{\rightarrowtail}
\newcommand{\epi}{\twoheadrightarrow}

\def\QQ{{\mathbb Q}}
\def\ZZ{{\mathbb Z}}
\def\GG{{\mathcal G}}
\def\MM{{\mathcal M}}

\begin{document}

\title{On Bousfield's problem for solvable groups of finite Pr\"ufer rank}
\author{Sergei O. Ivanov}
\address{Chebyshev Laboratory, St. Petersburg State University, 14th Line, 29b,
Saint Petersburg, 199178 Russia} \email{ivanov.s.o.1986@gmail.com}\maketitle
\begin{abstract} For a group $G$ and $R=\ZZ,\ZZ/p,\QQ$ we denote by $\hat G_R$ the $R$-completion of $G.$ We study the map $H_n(G,K)\to H_n(\hat G_R,K),$ where $(R,K)=(\ZZ,\ZZ/p),(\ZZ/p,\ZZ/p),(\QQ,\QQ).$ We prove that $H_2(G,K)\to H_2(\hat G_R,K)$ is an epimorphism for a finitely generated solvable group $G$ of finite Pr\"ufer rank. In particular, Bousfield's $HK$-localisation of such groups coincides with the $K$-completion for $K=\ZZ/p,\QQ.$ Moreover, we prove that $H_n(G,K)\to H_n(\hat G_R,K)$ is an epimorphism for any $n$ if $G$ is a finitely presented group of the form $G=M\rtimes C,$ where $C$ is the infinite cyclic group and $M$ is a $C$-module. 
\end{abstract}

\let\thefootnote\relax\footnote{The research is supported by the Russian Science Foundation grant
N 16-11-10073.}

\section*{Introduction}

Throughout the paper $R$ denotes one of the rings $\ZZ,\QQ,\ZZ/p$ and $K$ denotes one of the fields $\QQ$ or $\ZZ/p,$ where $p$ is prime. If $R$ and $K$ occur together, we assume that $K$ is a quotient of $R.$ In other words,  there are three variants for $(R,K):$
$$ (R,K)\in \{\ (\ZZ,\ZZ/p), \ (\ZZ/p,\ZZ/p), \ (\QQ,\QQ)\ \} .  $$

Bousfield in \cite{Bousfield75} introduces the $R$-homological localization $X_R$ of a space $X.$ It is a space $X_R$ with an $R$-homology equivalence $X\to X_R$, which is universal among all $R$-homology equivalences $X\to Y.$ He proves that the $R$-localization exists for any space. The $R$-localization is well understood for 1-connected spaces. However, if $X$ is not 1-connected, even the computation of $\pi_1(X_R)$ can be a very difficult problem. Bousfield reduces the problem to a purely algebraic problem introducing the $HR$-localization $G_{HR}$ of a group $G$. It is a group $G_{HR}$ together with a homomorphism $G\to G_{HR}$ that induces an isomorphism on the level of the first homology $H_1(G,R)\cong H_1(G_{RH},R),$  an epimorphism on the level of the second homology $H_2(G,R)\epi H_2(G_{HR},R)$ and which is universal among such homomorphisms $G\to H$. He proves that the $HR$-localization exists for any group and   
 $$\pi_1(X_R)=\pi_1(X)_{HR}.$$ 
 
The construction of the $HR$-localization is not explicit and it is still difficult to compute it. But we can compare the $HR$-localization with the  more explicit construction of $R$-completion of a group $G,$ which is defined as follows 
$$\hat G_R=\begin{cases} \varprojlim\ G/\gamma^{R}_iG,& \text{ \ if \ } R\in \{\ZZ, \ZZ/p\}\\
\varprojlim\ G/\gamma^{\QQ}_iG \otimes \QQ,& \text{ \ if \ } R=\QQ,
  \end{cases} $$
 where $\gamma^R_iG$ is the lower $R$-central series of $G$ and $-\otimes \QQ$ denotes the Malcev completion ($P$-localization, where $P$ is the set of all primes, in terms of \cite{Hilton}) of a nilpotent group.  There is a natural homomorphism 
\begin{equation}\label{eq_HR_to_R-compl}
G_{HR}\longrightarrow \hat G_R.
\end{equation} 
If the map \eqref{eq_HR_to_R-compl} is an isomorphism, the group $G$ is called $HR${\bf -good} (similarly to $R$-good spaces).  Otherwise, it is called $HR${\bf -bad}.  In the book of Bousfield  \cite{Bousfield77} there is only one problem:

\

{\bf Problem} (Bousfield): {\it Is there an $HK$-bad finitely presented group for $K=\QQ$ or $K=\ZZ/p$?} 

\

It is known that a group $G$ is $HR$-good if and only if the map $H_1(G,R)\to H_1(\hat G_R,R) $ is an isomorphism and $H_2(G,R)\to H_2(\hat G_R,R)$ is an epimorphism. If $G$ is finitely generated, then $H_1(G,R)\to H_1(\hat G_R,R)$ is an isomorphism (because $G_{HR}/\gamma_\omega^R(G_{HR})=\hat G_R$ \cite{Bousfield77}). So the problem of Bousfield can be reformulated as follows: Is there a finitely presented group $G$ such that $H_2(G,K)\to H_2(\hat G_K,K)$ is not an epimorphism for $K=\QQ$ or $K=\ZZ/p$? 

We are interested not only in $\hat G_K$ for $K=\ZZ/p,\QQ.$ The pronilpotent completion $\hat G_\ZZ$ is interesting as well. It is easy to find an example of an $H\ZZ$-bad finitely presented group (for example, Klein bottle group $\ZZ\rtimes C$).  However, the homomorphism on the level of homology with finite coefficients $H_2(G,\ZZ/p)\to H_2(\hat G_\ZZ,\ZZ/p)$ is usually an epimorphism for all examples that we know.  So we generalise the problem of Bousfield and formulate it in the form of conjecture.

\ \\
{\bf Conjecture:} {\it If $G$ is a finitely presented group, then the map 
\begin{equation}\label{eq_homol_map_2}
H_2( G,K)\longrightarrow H_2(\hat G_R,K)
\end{equation}
is an epimorphism for $(R,K)\in \{(\ZZ/p,\ZZ/p),(\ZZ,\ZZ/p),(\QQ,\QQ)\}$. }

\

In  \cite{Ivanov_Mikhailov} Roman Mikhailov and I proved that the conjecture is true for the class of metabelian groups. Generally we do not believe that the conjecture is true for all finitely presented groups. We tried to find a counterexample among those solvable groups, for which we can compute $H_2(\hat G_R,K).$ For example, for Abels' group \cite{Abels} (which is known as a counterexample for some conjectures that are true for metabelian groups) or for finitely presented groups of the form $N\rtimes C,$ where $C$ is the infinite cyclic group and $N$ is nilpotent. But in all these cases the map turns out be an epimorphism. Then I found a reason why in all these cases the map is an epmorphism. All these groups are of finite Pr\"ufer rank. The first goal of the paper is to study homological properties of $\hat G_R$ for solvable groups $G$ of finite Pr\"ufer rank and prove the conjecture for them. 

Recall that a group is of type $F_n$ if it has a classifying space, whose $n$-skeleton is finite. A group is of type $F_1$ if and only if it is finitely generated. A group is of type $F_2$ if and only if it is finitely presented. It seems, it is natural to look at the homomorphism $H_n(G,K)\to H_n(\hat G_R,K)$ for a group $G$ of type $F_n$ for any $n$ and ask, when it is an epimorphism. So the most general problem in this direction is the following. 
\\ \ \\
{\bf Problem:} {\it Describe groups $G$ of type $F_n$ such that the map 
\begin{equation}\label{eq_homol_map_n}
H_n(G,K) \longrightarrow H_n(\hat G_R,K)
\end{equation}
is an epimorphism. }

\

Bousfield proofs \cite{Bousfield92} that one of the groups $H_2(\hat F_{\ZZ/p},\ZZ/p),$ $H_3(\hat F_{\ZZ/p},\ZZ/p)$ is not trivial for a free group $F$, and hence, not for all groups of type $F_n$ the map \eqref{eq_homol_map_n} is an epimorphism. However, it seems the class of groups of type $F_n$ such that the map \eqref{eq_homol_map_n} is an epimorphism is very large. We denote by $C$ the infinite cyclic group. If $M$ is a $C$-module, we can consider the semidirect product $M\rtimes C.$ Bieri and Strebel describe all modules $M$ such that $M\rtimes C$ is finitely presented \cite[Th. C]{Bieri-Strebel_78}. It can be shown, and we will show this later, that if $M\rtimes C$ is finitely presented, then $M\rtimes C$ is of type $F_n$ for any $n.$
The second goal of the article is to prove that for any finitely presented metabelian group of the form 
$G=M\rtimes C$ the homomorphism \eqref{eq_homol_map_n} is an epimorphism for any $n$. 

Usually we look at the $R$-completion $\hat G_R$ as on a discrete group. But it turned out that for our purposes it is useful to look on this group as on a topological group with the topology of inverse limit. Unfortunately, the standard continuous cohomology of topological groups \cite{Stasheff} behave well enough only for  compact groups. Because of this, we introduce  {\bf unform cohomology} $H^*_u(\mathscr{G},K)$ of a topological group $\mathscr{G},$ which are defined similarly to continuous cohomology, but instead of continuous chains we consider uniform chains. Here we consider only topological groups, whose left and right uniform structures coincide. We like the uniform cohomology more then the continuous cohomology because they send inverse limits to direct limits in a more general case (see Proposition \ref{proposition_uniform_cohomology}). A topological group $\mathscr{G}$ is said to be $n${\bf -cohomologically discrete} over $K$ if the obvious map is an isomorphism 
$$H^m_u(\mathscr{G},K) \cong H^m(\mathscr{G},K)$$
for any $m\leq n.$ If a topological group is $n$-cohomologically discrete over $K$ for any $n,$ we call it just cohomologically discrete over $K$. This definition is interesting for us because of the following statement that we prove.

\

{\bf Proposition.} {\it Let $G$ be a finitely presented group. If $\hat G_R$ is $2$-cohomologically discrete over $K$, then the map \eqref{eq_homol_map_2} is an epimorphism.}

\

For more detailed version and the proof see Proposition \ref{prop_eqviv_2-coh-discr}. 

A group $G$ is of finite Pr\"ufer rank (special rank, reduced rank) if there is a number $r$ such that any finitely generated subgroup of $G$ can be generated by $r$ generators. The minimal $r=r(G)$ with this property is called Pr\"ufer rank of $G.$  If there is a short exact sequence of groups $G'\mono G \epi G'',$ then $G$ is of finite Pr\"ufer rank if and only if $G'$ and $G''$ are of finite Pr\"ufer rank and $r(G)\leq r(G')+r(G'')$ \cite[Lemma 1.44]{Robinson}. Let $A$ be an abelian group. For  prime $p$ we set $${}_pA=\{a\in A\mid pa=0\}.$$ The group $A$ is of finite Pr\"ufer rank if and only if $A\otimes \QQ$ is finite dimensional over $\QQ$, ${}_pA$ is finite dimensional over $\ZZ/p$ and the dimensions are bounded over all $p$  \cite[pp. 33-34]{Robinson}. In this case 
$$r(A)={\sf dim}_\QQ(A\otimes \QQ)+{\rm max}\{{\sf dim}_{\ZZ/p} ({}_pA)\mid p \text{ is prime}\}.$$ 
Obviously, a solvable group $G$ is of finite Pr\"ufer rank if and only if there is a finite sequence of normal subgroups $G=U_1\supseteq \dots \supseteq U_s=1$ such that $U_i/U_{i+1}$ is an abelian group of finite Pr\"ufer rank. For example, the group $\ZZ[1/2]\rtimes C,$ where $C=\langle t \rangle$ is the infinite cyclic group acting on $\ZZ[1/2]$ by the multiplication on $2,$ is a finitely generated non-polycyclic metabelian group of Pr\"ufer rank $2.$  We prove the following theorem.

\

{\bf Theorem.} {\it Let $G$ be a solvable group of finite Pr\"ufer rank. Then $\hat G_R$ is cohomologically discrete over $K$. Moreover, if $G$ is finitely generated, then the map \eqref{eq_homol_map_2} is surjective and $G$ is $HK$-good.}

\

For a more detailed version and the proof see Proposition \ref{proposition_solvable_Q-prenilpotent} and Theorem \ref{theorem_solvable_R-compl}. 

We denote by $C$ the infinite cyclic group. A finitely generated $C$-module $M$ is called tame if: 
(1) the torsion subgroup ${\sf tor}(M)$ is finite;
(2) $M\otimes \QQ$ is finite dimensional;
(3) there exists a generator $t$ of $C$ such that the characteristic polynomial of \hbox{$t\otimes \QQ \in {\rm GL}(M\otimes \QQ)$} is integral. 
Bieri and Strebel proved \cite[Th. C]{Bieri-Strebel_78} that the group $G=M\rtimes C$ is finitely presented if and only if the $C$-module $M$ is tame. Note that a tame $C$-module is of finite Pr\"ufer rank. Hence, $G=M\rtimes C$ is a metabelian group of finite Pr\"ufer rank. It follows from the main theorem of \cite{Aberg} that $G$ is of type $FP_n$ for any $n.$ By \cite{BieriStrebel80} we know that $FP_n$ implies $F_n.$ Therefore, if $M$ is a tame $C$-module, $M\rtimes C$ is of type $F_n$ for any $n.$ We prove the following theorem. 

\

{\bf Theorem.} {\it Let $M$ be a tame $C$-module and $G=M\rtimes C$. Then the map \eqref{eq_homol_map_n} is an epimorphism for any $n$.}

\

For more detailed version see Proposition \ref{proposition_MrtimesC} and Theorem \ref{theorem_MrtimesC}. Moreover, we prove the following. If $C=\langle t \rangle$ and $M(t-1)\subseteq pM,$ then the   map \eqref{eq_homol_map_n} gives an isomorphism $H_*(G,\ZZ/p)\cong H_*(\hat G_{\ZZ/p},\ZZ/p) $ and it follows that the $\ZZ/p$-localization of the classifying space is the classifying space of the $\ZZ/p$-completion 
\begin{equation}\label{eq_bg_iso}
(BG)_{\ZZ/p}\cong B(\hat G_{\ZZ/p}).
\end{equation}
In particular, the isomorphism \eqref{eq_bg_iso} holds in the case $G=\ZZ^n\rtimes_a C,$ where $C$ acts on $\ZZ^n$ by a matrix $a\in {\rm GL}_n(\ZZ)$ such that entries of $a-1$ are divisible by $p.$

A lot of things can be done much simpler here, if we assume that $K\ne \ZZ/2 \ne R$ and use the natural isomorphism 
$$H_*(A,\ZZ/p)\cong \Lambda^*(A/p) \otimes \Gamma^{*/2}({}_pA)$$
for an abelian group $A$ and $p\ne 2$ (see \cite[Ch. V, Th. 6.6]{Brown}, \cite{Cartan}). But we intentionally do not use it because we want to cover the case of $K=\ZZ/2.$

The paper is organised as follows. In Section \ref{sectionAlemmaaboutspectralsequences} we prove some lemmas that generalise the standard fact that, if a morphism of converging spectral sequences is an isomorphism on some page, then it converges to an isomorphism.  In Section  \ref{sectionCohomologyofinverselimits} we study inverse systems of groups. Namely, we study relations between cohomology of the inverse limit, the direct limit of cohomology and the uniform cohomology of the inverse limit.  In section \ref{sectionMod-phomology} we study mod-$p$ homology of a completion of a solvable group of finite Pr\"ufer rank with respect to any filtration. In Section \ref{sectionRcompletions} we study $R$-completions of groups. In particular, we prove here that, if $R$-completion is $2$-cohomologically discrete, then the map $H_2(G,K)\to H_2(\hat G_R,K)$ is an epimorphism, and the main result that the $R$-completion of a solvable group of finite Pr\"ufer rank is cohomologically discrete over $K$. In Section \ref{sectionmetabliangroupsoverC} we prove that, for a tame $C$-module $M,$ if we set $G=M\rtimes C,$ then $H_n(G,K)\to H_n(\hat G_R,K)$ is an epimorphism. 

\section{Lemmas about spectral sequences}\label{sectionAlemmaaboutspectralsequences}

\begin{Lemma} Let $f:E\to E'$ be a morphism between two spectral sequences of cohomological type converging to a morphism of graded abelian groups $\varphi:H\to H'$ and $r,n$ be natural numbers. Assume that the spectral sequences are concentrated in the first quadrant. If for any pair of integers $(k,l)$ such that $(r-1)k\leq r(n-l)$ the morphism on the level of $r$th pages
 $f_r^{k,l}:E_r^{k,l}\to {E'}_r^{k,l}$ is an isomorphism, 
 then $\varphi^m:H^m\to {H'}^m$ is an isomorphism for $m\leq n.$
\end{Lemma}
\begin{proof} Fix $n.$ Note that for any pair $(k,l)$  such that $k+l\leq n$ and any $r$ the inequality $(r-1)k\leq r(n-l) $ holds. First, assume that $r\geq n+2.$ In this case, if $k+l\leq n,$ then $E_r^{k,l}=E_\infty^{k,l}$ and ${E'}_r^{k,l}={E'}_\infty^{k,l}.$ It follows that $\varphi^m$ is an isomorphism for $m\leq n.$ Prove now the statement for $r\leq n+2$ by induction on $n+2-r.$ We already proved the base step $n+2-r=0.$ Prove the inductive step. The direct sum $\bigoplus_{(k,l):(r-1)k\leq r(n-l)} E_r^{k,l}$ is a direct summand of the differential module $(E_r,d_r)$ because the bidegree of $d_r$ is $(r,-r+1)$ and, if $(k,l)$ satisfies $(r-1)k\leq r(n-l),$ then $(k+r,l-r+1)$ satisfies $(r-1)(k+r)\leq r(n-(l-r+1)).$  The same for $E'$ and $f_r:E_r\to E'_r$ respects this decomposition. It follows that $f_r$ induces an isomorphism on $(r+1)$st page $f_{r+1}^{k,l}:E_r^{k,l}\to {E'}_r^{k,l}$ for  indexes $(k,l)$ such  that $ (r-1)k\leq r(n-l).$ In particular, $f_{r+1}^{k,l}$ is an isomorphism for the smaller set of indexes satisfying
$ ((r+1)-1)k\leq (r+1)(n-l).$ The assertion follows from the inductive hypothesis.     
\end{proof}

There is a dual version. Formally it does not follow from the previous, but the proof is the same.

\begin{Lemma} Let $f:E\to E'$ be a morphism between two spectral sequences of homological type converging to a morphism of graded abelian groups $\varphi:H\to H'$ and $r,n$ be natural numbers. Assume that the spectral sequences are concentrated in the first quadrant. If for any pair of integers $(k,l)$ such that $(r-1)k\leq r(n-l)$ the morphism on the level of $r$th pages
 $f^r_{k,l}:E^r_{k,l}\to {E'}^r_{k,l}$ is an isomorphism, 
 then $\varphi_m:H_m\to {H'}_m$ is an isomorphism for $m\leq n.$
\end{Lemma}
\begin{proof}
The proof is the same.
\end{proof}

There are two corollaries for the second page.

\begin{Corollary}\label{cor_spectral_seq_coh} Let $f:E\to E'$ be a morphism between two spectral sequences of cohomological type converging to a morphism of graded abelian groups $\varphi:H\to H'$ and $n$ be a natural number. Assume that the spectral sequences are concentrated in the first quadrant. If for any pair of integers $(k,l)$ such that $k\leq 2(n-l)$ the morphism 
 $f_2^{k,l}:E_2^{k,l}\to {E'}_2^{k,l}$ is an isomorphism, 
 then $\varphi^m:H^m\to {H'}^m$ is an isomorphism for $m\leq n.$
\end{Corollary}

\begin{Corollary} Let $f:E\to E'$ be a morphism between two spectral sequences of homological type converging to a morphism of graded abelian groups $\varphi:H\to H'$ and $n$ be a natural number. Assume that the spectral sequences are concentrated in the first quadrant. If for any pair of integers $(k,l)$ such that $k\leq 2(n-l)$ the morphism 
 $f^2_{k,l}:E^2_{k,l}\to {E'}^2_{k,l}$ is an isomorphism, 
 then $\varphi_m:H_m\to {H'}_m$ is an isomorphism for $m\leq n.$
\end{Corollary}

\section{Cohomology of inverse limits}\label{sectionCohomologyofinverselimits}

Throughout the section we denote by $P$ a fixed filtered  partially ordered set (=directed set). The word ``filtered'' means that for any $i,j\in P $ there exists $k\in P$ such that
 $k\geq i$ and $k\geq j.$ The main example of $P$ for us is $\mathbb N$ but occasionally we use other examples like all finitely generated subgroups of a group.  
By a directed system in a category $\mathfrak{C}$ we mean a functor $\mathcal C:P\to \mathfrak{C},$ where $P$ is treated as a category. 
For $i\in P $ we set $C^i:=\mathcal C(i).$ 
Dually we define an inverse system $\mathcal C:P^{\rm op}\to \mathfrak{C}.$ 
In the case of inverse systems we use the dual notation $C_i:=\mathcal C(i).$ We say that a property holds for big enough $i$ if there exists $k\in P$ such that the property holds for any $i\geq k.$ 

We use directed limits (=filtered colimits) $\varinjlim$ of directed systems and inverse limits $\varprojlim$ of inverse systems. The important advantage of directed limits in contrast to inverse limits is that $\varinjlim: {\sf Ab}^P \to {\sf Ab}$ is an exact functor. 

Throughout the section $\GG$ denotes an inverse system groups and {\bf epimorphims}. We set $G_i:=\GG(i).$ So we assume that the homomorphism $G_j\to G_i$ is an epimorphism for any $i\leq j.$

Let $S$ be a commutative ring. A {\bf direct  $S[\GG]$-module} $\MM$ is a direct system of abelian groups,  together with a structure of $S[G_i]$-module on $M^i$  such that $M^i\to M^{j}$ is a homomorphism of $S[G_{j}]$-modules for any $i\leq j$. In the case $S=\ZZ$ we call them just direct $\GG$-modules.
If we set
$$ \hat \GG:=\varprojlim G_i,\hspace{1cm} \breve \MM=\varinjlim M^i,$$
then $\breve \MM$ has a natural structure of a $S[\hat \GG]$-module. Note that if $\mathcal G'\mono \mathcal G \epi \mathcal G''$ is a short exact sequence of groups and epimorphisms, then by Mittag-Leffler condition the sequence of inverse limits $\hat \GG'\mono \hat \GG \epi \hat\GG''$ is short exact. 

\subsection{Cohomologically discrete direct modules}

A direct $\GG$-module $\MM$ is said to be $n$-cohomologically discrete, if the homomorphism 
$$\varinjlim H^m(G_i,M^i)\longrightarrow H^m(\hat \GG,\breve \MM)$$
is an isomorphism for any $m\leq n$ (further we will prove that this definition is equivalent to a definition given on the language of uniform cohomology for some class of direct modules Proposition \ref{cor_quasiconstant_equiv}).
The inverse system $\GG$ is said to be $n$-cohomologically discrete over $K$ if the constant $\GG$-module $K$ is $n$-cohomologically discrete. If $\GG$ is $n$-cohomologically discrete $K$ for any $n,$ it is called cohomologically discrete over $K.$

\begin{Proposition} \label{prop_direct_coh_discr} Let $\GG'\mono \GG\epi \GG''$ be a short exact sequence of inverse systems of groups and epimorphisms and $\MM$ be a direct $\GG$-module.
Assume that 
\begin{enumerate}
\item $\MM$ is $n$-cohomologically discrete as a direct $\GG'$-module;
\item the direct $\GG''$-module $H^m(\GG',\MM)= \{H^m(G'_i,M^i)\}$ is $2(n-m)$-cohomologically discrete for any $0\leq m\leq n.$
\end{enumerate}
Then $\MM$ is $n$-cohomologically discrete as a direct $\GG$-module. 
\end{Proposition}
\begin{proof}
For any fixed $i$ the morphism of short exact sequences 
$$
\begin{tikzcd} 1\arrow[r]& \hat \GG' \arrow[r]\arrow[d] & \hat \GG \arrow[r]\arrow[d] & \hat \GG''\arrow[r] \arrow[d] &1\\
1 \arrow[r]& G_i'\arrow[r] & G_i\arrow[r] & G_i''\arrow[r] & 1
\end{tikzcd}
$$
and the homomorphism of $\hat \GG$-modules $\hat \MM\to M^i$
yield a morphism of Lyndon–Hochschild–Serre spectral sequences, whose second page component is 
$$H^*(G_i'',H^*(G_i',M^i)) \longrightarrow H^*(\hat \GG'',H^*(\hat \GG',\breve M))$$
and 
that converges to the morphism $H^*(G_i,M^i)\to H^*(\hat \GG,\breve \MM).$ On the left hand we have a direct system of spectral sequences. Since direct limit is an exact functor, we obtain the spectral sequence of direct limits, whose second page is $\varinjlim H^*(G_i'',H^*(G_i',M^i))$ and that converges to $\varinjlim H^*(G_i,M^i).$ Moreover, we obtain a morphism of spectral sequences, whose second page component is
$$ \varinjlim H^*(G_i'',H^*(G_i',M^i)) \longrightarrow H^*(\hat \GG'',H^*(\hat \GG',\breve \MM))$$ and that converges to the homomorphism $\varinjlim H^*(G_i,M^i)\to H^*(\hat \GG,\breve \MM).$ Assume that $(k,l)$ is a pair of non negative integers such that $k\leq 2(n-l).$ Then 
$$\varinjlim H^k(G_i'',H^l(G_i',M^i))\longrightarrow  H^k(\hat \GG'',\varinjlim H^l(G_i',M^i))$$ 
is an isomorphism because $H^l(\GG',\MM)$ is $2(n-l)$-cohomologically discrete. 
The map $\varinjlim H^l(G'_i,M^i)\to H^l(\hat \GG',\breve \MM)$ is an isomorphism because $\MM$ is $n$-homologically discrete as a direct $\GG'$-module. 
It follows that 
$$\varinjlim H^k(G_i'',H^l(G_i',M^i)) \longrightarrow H^k(\hat \GG'',H^l(\hat \GG',\breve \MM))$$ is an isomorphism for pairs $(k,l)$ such that $k\leq 2(n-l).$ The assertion follows from Corollary \ref{cor_spectral_seq_coh}.
\end{proof}

\begin{Corollary} Let $\GG'\mono \GG\epi \GG''$ be a short exact sequence of inverse systems of groups and epimorphisms.
Assume that 
\begin{enumerate}
\item  $\GG'$ is $n$-cohomologically discrete over $K$;
\item the direct $\GG''$-module $H^m(\GG',K)= \{H^m(G'_i,K)\}$ is $2(n-m)$-cohomologically discrete for any $0\leq m\leq n.$
\end{enumerate}
Then $\GG$ is $n$-cohomologically discrete. 
\end{Corollary}

\subsection{Quasiconstant direct systems}

Let $\mathfrak{A}$ be an abelian category. A direct system $\mathcal A$ in $\mathfrak{A}$
is said to be {\bf zero-equivalent} if for any $i$ there is $j>i$ such that the morphism $A^i\to A^j$ vanishes. In this case $\varinjlim A^i$ exists and vanishes.

Let $\mathcal A$ be a direct system in 
$\mathfrak{A}$. 
Assume that $\breve{\mathcal A}=\varinjlim A^i$ exists. 
Consider the kernels $B^i={\rm Ker}(A^i\to \breve {\mathcal A}).$ They form a direct system of groups $\mathcal B$. 
Then $\mathcal A$ is called {\bf quasiconstant} if 
$\mathcal B$ is zero-equivalent and the morphism 
$A^i\to \breve{\mathcal A}$ is an epimorphism for big enough $i$. 
Equivalently, $\mathcal A$ is quasiconstant if there exists a short exact sequence of direct systems $\mathcal B\mono \mathcal A \epi \mathcal C$ such that $\mathcal B$ is zero-equivalent and $C^i \to C^{j}$ is an isomorphism for big enough $i$ and any $j\geq i.$ In this case $ \breve A\cong C^i$ for  big enough $i.$

Dually one can define  zero-equivalent inverse systems and quasiconstant inverse systems. 

\begin{Lemma}\label{lemma_contravarian_functor} Let $\mathfrak{F}:\mathfrak{A}^{\rm op}\to \mathfrak{B}$ be an exact contravariant functor from an abelian category $\mathfrak{A}$ to an abelian category $\mathfrak{B}$. If $\mathcal A$ is a quasiconstant inverse system in $\mathfrak{A}$, then $\mathfrak{F}(\mathcal A)$ is a quasiconstant direct system in $\mathfrak{B}$ and
$$\varinjlim \mathfrak{F}(A_i)=\mathfrak{F}(\varprojlim A_i).$$ 
\end{Lemma}
\begin{proof}
Obvious. 
\end{proof}

\begin{Lemma}\label{lemma_direct_system_of_finit_length} Let $\mathcal A$ be an inverse system (resp. directed system) of $S$-modules of finite length over $S.$ Assume that the lengths of the modules are bounded above i.e. there exists $n$ such that ${\sf length}(A_i)\leq n$ (resp. ${\sf length}(A^i)\leq n$) for any $i.$ Then $\mathcal A$ is quasiconstant and ${\sf length}(\varprojlim A_i)\leq n.$
\end{Lemma}
\begin{proof}
For any fixed $i$ the system of images ${\rm Im}(A_{j}\to A_i)$ for $j\geq i$ stabilises because $A_i$ has finite length. In other words there exists  $C_i\subseteq A_i$ such that $C_i={\rm Im}(A_{j}\to A_i)$ for big enough $j\geq i.$ Note that the inverse system $A_i/C_i$ is zero-equivalent.  It is easy to see that $C_{i+1}\to C_i$ is an epimorphism. Hence, ${\sf length}(C_{j})\geq {\sf length}(C_i)$ for $j\geq i.$ Since the lengths are bounded above, $C_{j}\to C_i$ is an isomorphism for  big enough $i.$  

The proof for directed systems is similar. 
\end{proof}

A direct $S[\GG]$-module is called quasiconstant if it is qasiconstant as a direct system of abelian groups.  

\begin{Lemma} Let $\MM$ be a direct $S[\GG]$-module such that each $M^i$ is a Noetherian $S[G_i]$-module. Then the following are equivalent. 
\begin{enumerate}
\item $\breve\MM$ is a Noetherian $S[\hat \GG]$-module.
\item $\MM$ is quasiconstant.
\end{enumerate}
\end{Lemma}  
\begin{proof}

(2)$\Rightarrow$(1) is obvious because $\breve \MM$ is a quotient of $M^i$ for  big enough $i.$

(1)$\Rightarrow$(2). Since $\breve \MM=\bigcup_{i=1}^\infty {\rm Im}(M^i\to \breve \MM)$ and $\breve \MM$ is Noetherian, we obtain ${\rm Im}(M^i\to \breve \MM)=\breve M$ for  big enough $i.$ Hence $M^i\to \breve \MM$ is an epimorphism for  big enough $i.$ Without loss of generality we can further assume that it is an epimorphism for all $i.$ Set $N^i={\rm Ker}(M^i\to \breve \MM).$ Then we have a short exact sequence of direct modules $\mathcal N\mono \mathcal M \epi \breve \MM.$ It follows that $\varinjlim N^i=0.$ Hence for any fixed $i$ we have $\bigcup_{j>i}^\infty {\rm Ker}(N^i\to N^{j})=N^i.$ Using that $M^i$ is Noetherian, we obtain $N^i={\rm Ker}(N^i\to N^j)$ for  big enough $j.$
\end{proof}

\subsection{Cohomologically discrete inverse limits}

Let $\mathscr{G}$ be a topological group, whose right and left uniform structures coincide. For a topological $\mathscr{G}$-module $\mathscr{M}$ we define {\bf uniform cohomology} of $\mathscr{G}$ with coefficients in $\mathscr{M}$ 
$$H_{u}^n(\mathscr{G},\mathscr{M})$$
similarly to continuous cohomology \cite{Stasheff} but instead of the complex of continuous cochains we consider the complex of uniform cochains 
$$C_u^n(\mathscr{G},\mathscr{M})= \{f:\mathscr{G}^n\to \mathscr{M}\mid f \text{ is uniform }\}.$$   
If $\mathscr{G}$ is a compact topological group, uniform cohomology coincide with continuous cohomology.

A discrete $\mathscr{G}$-module $\mathscr{M}$ is said to be $n${\bf -cohomologically discrete} if the morphism 
$$H^m_u(\mathscr G,\mathscr M)\longrightarrow H^m(\mathscr G,\mathscr M)$$
is an isomorphism for $m\leq n.$ The topological group $\mathscr{A}$ is said to be $n$-cohomologically discrete over a commutative ring $S$ if $S$ is a $n$-cohomologically discrete $\mathscr{G}$-module.

\begin{Proposition}\label{proposition_uniform_cohomology} Let  $\MM$ be a quasiconstant direct $\GG$-module. Consider $\hat \GG$ as a topological group with the topology of inverse limit and $\breve \MM$ as a discrete $\hat \GG$-module. Then $$H_u^*(\hat \GG,\breve \MM)=\varinjlim\: H^*(G_i,M^i).$$ 
\end{Proposition}  
\begin{proof} Without loss of generality, we can assume that $M^i\to \breve M$ is an epimorphism for any $i.$ Set $N^i:={\rm Ker}(M^i\to \breve M).$ The short exact sequence $N^i\mono M^i\epi \breve \MM$ gives rise a long exact sequence
$$H^{n-1}(G_i,N^i)\to H^n(G_i,\breve \MM) \to H^n(G_i,M^i) \to H^n(G_i,N^i).$$ Since $\mathcal N$ is zero-equivalent, $H^*(G_i,N_i)$ 
is zero-equivalent as well. Since the functor of direct limit is exact, it follows that $\varinjlim H^n(G_i,\breve \MM) \cong \varinjlim H^n(G_i,M^i).$ 

 Now we need to prove that 
$H^*(\hat \GG,\breve \MM) = \varinjlim H^*(G_i,\breve \MM).$ Since the functor of direct limit is exact, it is enough to prove  that $C_u^n(\hat \GG,\breve \MM)=\varinjlim\: C^n(G_i,\breve \MM).$ Consider the morphism $\varinjlim \: C^n(G_i,\breve \MM) \to C_u^n(\hat \GG,\breve \MM).$ Since the maps $\hat \GG\to G_i$ are epimorphisms, we get that it is a monomorphism. In order to prove that it is an epimorphism, we need to prove that for any $f\in C_u^n(\hat \GG,\breve \MM)$ there exists $i$ and  $\bar f\in C^n(G_i,\breve \MM)$ such that $f$ is a composition of $\bar f$ and the projection $\hat \GG^n\to G_i^n$. Set $H_i:={\rm Ker}(\hat \GG\to G_i).$ Then the uniform structure on $\hat \GG$ is generated by sets  $U_i=\{(g,g')\mid gH_i=g'H_i\}.$ It follows that for any $f$ there exists $i$ such that $f(g_1,\dots,g_n)=f(g_1',\dots,g_n')$ whenever $g_kH_i=g_k'H_i$ for any $k$. Then, if we identify $G_i$ with $\hat \GG/H_i,$ we can define $\bar f(g_1H_i,\dots,g_nH_i)=f(g_1,\dots,g_n).$
\end{proof}

\begin{Corollary}\label{cor_quasiconstant_equiv} Let $\MM$ be a quasiconstant direct $\GG$-module. Then $\MM$ is $n$-cohomologically discrete as a direct $\GG$-module if and only if $\breve \MM$ is $n$-cohomologically discrete as a discrete $\hat \GG$-module. 
\end{Corollary}

\begin{Proposition} \label{proposition_cohom_discr_mod}
Let $\GG'\mono \GG\epi \GG''$ be a short exact sequence of inverse systems of groups and epimorphisms and $\MM$ be a quasiconstant direct $\GG$-module.
Assume that 
\begin{enumerate}
\item $\breve \MM$ is $n$-cohomologically discrete as a discrete $\hat\GG'$-module;
\item the discrete $\hat \GG''$-module $H^m(\hat \GG',\breve \MM)$ is $2(n-m)$-cohomologically discrete for any \hbox{$0\leq m\leq n.$}
\item the direct system $H^m(G'_i,M^i)$ is quasiconstant for any \hbox{$0\leq m\leq n.$}
\end{enumerate}
Then $\breve \MM$ is $n$-cohomologically discrete as a discrete $\hat \GG$-module. 
\end{Proposition}
\begin{proof}
By Corollary \ref{cor_quasiconstant_equiv} we obtain that $\MM$ is $n$-cohomologically discrete as a direct $\GG'$-module. Then by Proposition \ref{prop_direct_coh_discr}  and the Corollary \ref{cor_quasiconstant_equiv} it is enough to prove that $\{H^m( G'_i,M^i)\}$ is $2(n-m)$-cohomologically discrete direct $\GG''$-module for any $0\leq m\leq n.$ Fix $k\leq 2(n-m).$ Then we need to prove that $ \varinjlim H^k(G''_i, H^m(G'_i,M^i))\cong H^k(\hat \GG'',\varinjlim H^m(G'_i,M^i)).$ 
We have
$$\varinjlim H^k(G''_i, H^m(G'_i,M^i))\cong H^k_u(\hat \GG'', \varinjlim H^m(G'_i,M^i) )   
$$
because $H^m(G'_i,M^i)$ is quasiconstant. We have
$$H^k_u(\hat \GG'', \varinjlim H^m(G'_i,M^i) ) \cong H^k_u(\hat \GG'',  H^m(\hat \GG',\breve \MM))  $$
because $\MM $ is $n$-cohomologically discrete over $\GG'.$ We have
$$H^k_u(\hat \GG'',  H^m(\hat \GG',\breve \MM))\cong H^k(\hat \GG'',  H^m(\hat \GG',\breve \MM))$$
because $ H^m(\hat \GG',\breve \MM)$ is $2(n-m)$-cohomologically discrete as a discrete $\hat \GG''$-module. And using again that $\MM $ is $n$-cohomologically discrete over $\GG'$ we get 
$$H^k(\hat \GG'',  H^m(\hat \GG',\breve \MM))\cong H^k(\hat \GG'',  \varinjlim H^m( G'_i,M^i)).$$
Composing these isomorphisms we obtain that  $\{H^m( G'_i,M^i)\}$ is $2(n-m)$-cohomologically discrete direct $\GG''$-module for any \hbox{$0\leq m\leq n.$}
\end{proof}

\section{Mod-p homology of completions of abelian and solvable groups\\ of finite Pr\"ufer rank} \label{sectionMod-phomology}

For an abelian group $A$ we set $${}_nA=\{a\in A\mid na=0\},\hspace{1cm} {}_{(n)}A= \bigcup_{i=1}^\infty {}_{n^i}A,$$
$$ {\sf tor}(A)=\bigcup_{n=1}^{\infty} {}_nA, \hspace{1cm} A/n=A/nA.$$ 
We freely use that for a torsion free abelian group $A$ there are isomorphisms $H_*(A)\cong \Lambda^*(A)$ and $H_*(A,\ZZ/p)\cong \Lambda^*(A/p).$ If we set $V^\vee={\sf Hom}_{\ZZ/p}(V,\ZZ/p),$ then $H^*(A,M^\vee)=H_*(A,M)^\vee$ for a $\ZZ/p[A]$-module $M.$ If $M$ is finite,  $(M^\vee)^\vee=M.$ Then for a finite $\ZZ/p[A]$-module there is an isomorphism $H^*(A,M)=H_*(A,M^\vee)^\vee.$ Using this, one can obtain cohomological versions of all homological statements in the section.  

\subsection{Lemmas about abelian groups}

\begin{Lemma}\label{lemma_f_induces_homology} Let $f:A\to B$ be a homomorphism of abelian groups. 
\begin{enumerate}
\item If $f$ induces isomorphisms $A/p\cong B/p$ and ${}_pA\cong {}_pB$, then it induces an isomorphism $H_*(A,\ZZ/p)\cong H_*(B,\ZZ/p).$
\item If $f$ induces trivial maps $A/p\xrightarrow{0}B/p$ and ${}_pA\xrightarrow{0}{}_pB,$ then the map $H_n(A,\ZZ/p)\xrightarrow{0} H_n(B,\ZZ/p)$
 is trivial for $n\geq 1$. 
 \end{enumerate}
\end{Lemma}
\begin{proof}
(1) First assume that $B=0$ and prove that $H_n(A,\ZZ/p)=0$ for $n\geq 1.$ Consider the torsion subgroup ${\sf tor}(A)\subseteq A$ and the torsion free quotient ${\sf tf}(A)=A/{\sf tor}(A).$ Then ${\sf tf}(A)/p=0,$ and hence, $H_*({\sf tf}(A),\ZZ/p)=\Lambda^*({\sf tf}(A)/p)=0.$ Since $ {}_pA=0,$ we have ${}_{(p)}A=0,$ and then ${\sf tor}(A)$ consists of non-$p$-torsion. Then $H_n({\sf tor}(A),\ZZ/p)=0.$ Then the second page of the spectral sequence $H_*({\sf tf}(A),H_*({\sf tor}(A),\ZZ/p))$ consists of zeros except the position $(0,0).$

If $f$ is a monomorphism, the cokernel $D$ has the properties ${}_pD=0$ and $D/p=0$ because of the exact sequence $0\to {}_pA \to {}_pB\to {}_pD\to A/p \to B/p \to D/p\to 0.$ Then $H_n(D,\ZZ/p)=0$ for $n\geq 1,$ and using the spectral sequence $H_*(D,H_*(A,\ZZ/p))\Rightarrow H_*(B,\ZZ/p)$ we obtain the result. 

The proof is similar if $f$ is an epimorphism. 

Consider the general case. Denote by $I$ the image of $f.$ Since $f$ induces isomorphisms ${}_pA\cong {}_pB$ and $A/p\cong B/p,$ the maps ${}_pA\to {}_pI$ and $A/p\to I/p$ are monomorphisms and the maps ${}_pI\to {}_pB$ and $I/p\to B/p$ are epimorphisms. Using that $-\otimes \ZZ/p$ is right exact and ${\sf Tor}(-,\ZZ/p)$ is left exact, we obtain that $A/p\to I/p$ and ${}_pI\to {}_pB$ are isomorphisms. Then all these four maps are isomorphisms and, using that we proved above, we get $H_*(A,\ZZ/p)\cong H_*(I,\ZZ/p)\cong H_*(B,\ZZ/p).$

(2) Denote by $P$ the set of couples of finitely generated abelian groups $(A',B')$ such that $A'\leq A,$ $B'\leq B$ and $ f(A')\leq pB'.$ Consider $P$ with the natural filtered order. Obviously we have $B=\bigcup_{(A',B')\in P} B'.$ Prove that $A=\bigcup_{(A',B')\in P} A'$. Indeed, for any finitely generated subgroup $A'\leq A$ generated by $a_1,\dots,a_m$ we have $f(a_i)\in pB$ because the map $A/p\to B/p$ is trivial. Hence, we can find $b_1,\dots,b_m\in B$ such that $pb_i=f(a_i)$ and consider the subgroup $B'$ generated by them. Therefore $(A',B')\in P$ and $A=\bigcup_{(A',B')\in P} A'.$ Then $A=\varinjlim_{(A',B')\in P} A'$ and $B=\varinjlim_{(A',B')\in P} B'.$ So, using that homology commute with direct limits, we deduce the statement to finitely generated abelian groups. 

Prove the statement for finitely generated abelian groups. Using the functorial isomorphism $H_1(A,\ZZ/p)=A/p$ we obtain the result for $n=1.$ Prove it for $n\geq 2.$ Using the K\"unneth formula, we obtain that it is enough to prove this for three cases $A=\ZZ,$ $A=\ZZ/p^k$ and $A=\ZZ/q^k$ for $q\ne p.$ For $A=\ZZ/q^k$ and $A=\ZZ$ it is obvious because $H_n(\ZZ/q^k,\ZZ/p)=0$ and $H_n(\ZZ,\ZZ/p)$ for $n\geq 2.$ 

Consider the case $A=\ZZ/p^k.$ 
The assumptions on $f$ imply that $f$ factors through the homomorphism 
$\varphi:\ZZ/p^k \to \ZZ/p^k$ 
of multiplication by  $p:$ $\varphi(a)=pa.$  
Then it is enough to prove that $\varphi_* : H_n(\ZZ/p^k,\ZZ/p)\to H_n(\ZZ/p^k,\ZZ/p)$ is trivial for $n\geq 2.$  Consider the functorial short exact sequence $\Lambda^2(A/p) \mono H_2(A,\ZZ/p) \epi {}_pA.$ 
Using that $\Lambda^2(\ZZ/p)=0,$  that $\varphi$ induces the trivial map ${}_p(\ZZ/p^k)\to {}_p(\ZZ/p^k)$ and that the sequence is functoral, we obtain the result for  $n=2.$ 
Since homology in our case are dual to cohomology, it is enough to prove for cohomology. The map $\varphi: H^*(\ZZ/p^k,\ZZ/p)\to H^*(\ZZ/p^k,\ZZ/p)$ is an algebra homomorphism. For odd $p$ or $p=2$ and $k>1$ we have isomorphism of algebras $H^*(\ZZ/p^k,\ZZ/p)\cong \ZZ/p[x,y]/(x^2) ,$ 
and for $p=2$ for $k=1$ we have $H^*(\ZZ/2,\ZZ/2)\cong \ZZ/2[x],$ where $|x|=1$ and $|y|=2.$ Then $H^*(\ZZ/p^k,\ZZ/p)$ is generated  in the first and the second degree. Then, using that the maps $H^n(\ZZ/p^k,\ZZ/p)\to H^n(\ZZ/p^k,\ZZ/p) $ are trivial for $n=1,2,$ we obtain that they are trivial for all $n.$
\end{proof}

\begin{Lemma}\label{lemma_p_limits}  Let $A_1\leftarrow A_2 \leftarrow
 \dots $ be an inverse sequence of abelian groups and epimorphisms such that $A_i/p$ is finite for any $i$. Then the obvious maps are isomorphisms 
 $${}_p(\varprojlim A_i)\cong \varprojlim ({}_pA_i),\hspace{1cm} (\varprojlim A_i)/p \cong  \varprojlim (A_i/p).$$
\end{Lemma}
\begin{proof}
It follows from \cite[Prop. 2.4.]{Ivanov_Mikhailov} and the Mittag-Leffler condition.
\end{proof}

\begin{Lemma}\label{lemma_commplete_and_finite} Let $A$ be an abelian group such that $A/p$ is finite for any $p$ and $A\supseteq A_1 \supseteq A_2 \supseteq \dots $ be a sequence of subgroups. Set $\hat A=\varinjlim A/A_i$ and $\hat A_i={\rm Ker}(\hat A \to A/A_i).$ Then for any $n$ there exists $i$ such that $$\hat A_i\subseteq n\hat A.$$
\end{Lemma}
\begin{proof}
First we prove it for prime $n=p.$  By Lemma  \ref{lemma_p_limits} we get that $\hat A/p = \varprojlim_i\ (A/A_i)/p.$ Moreover, ${\sf dim}((A/A_i)/p)\leq {\sf dim}(A/p),$ and hence, the sequence stabilises: 
  $(A/A_i)/p = (\hat A/\hat A_i)/p=\hat A/p$ for big enough $i.$ Thus $p\hat A=p\hat A+\hat A_i $ for big enough $i.$ It follows that for any $p$  there is $i$ such that $\hat A_i \subseteq p\hat A.$ 
  
Now we prove the general case. Let $n=p_1\dots p_l,$ where $p_1,\dots,p_l$ are primes. Prove by induction on $l.$ We already proved the base step. Set $n'=p_2\dots p_l.$ Assume we proved the statement for $n'.$ Then there is $i$ such that $A_i\subseteq n'\hat A.$ Set $A'=n'A$ and $A'_j=A'\cap A_j.$ It is easy to see that $n'\hat A=\widehat{A'}.$ Then there is $j$ such that $A_j= \hat A'_j\subseteq p_1\widehat{A'}=n\hat A.$   
\end{proof}

\subsection{Lemmas about abelian groups of finite Pr\"ufer rank}

Recall that an abelian group is of finite Pr\"ufer rank if and only if the $\QQ$-vector space $A\otimes \QQ$ is finite dimensional, the $\ZZ/p$-vector space ${}_pA$ is finite dimensional for any $p$ and the dimensions are bounded above for all $p$ \cite[pp. 33-34]{Robinson}. In other words, $${\sf dim}_\QQ(A\otimes \QQ)+ {\rm max}\{ {\sf dim}_{\ZZ/p} ({}_pA)\mid p \text{ is prime } \}<\infty.$$  

For for an abelian group $A$ of finite Pr\"ufer rank  we set $$d_\QQ(A)={\sf dim}_\QQ(A\otimes \QQ), \hspace{1cm} d_p(A)={\sf dim}_{\ZZ/p}({}_pA), \hspace{1cm} D_p(A)=d_\QQ(A)+d_p(A).$$

\begin{Lemma}\label{lemma_D(A)} Let $A$ be an abelian group of finite Pr\"ufer rank. Then ${\sf dim}_{\ZZ/p}(A/p)\leq D_p(A)$ and for any subgroup $B\leq A$  $$ D_p(B)\leq D_p(A), \hspace{1cm} D_p(A/B)\leq D_p(A).$$
\end{Lemma}
\begin{proof} The inequality $D_p(B)\leq D_p(A)$ is obvious because $d_\QQ(B)\leq d_\QQ(A)$ and $d_p(B)\leq d_p(A).$ We set $E=A/B$ and prove $D_p(E)\leq D_p(A)$ and ${\sf dim}(A/p)\leq D_p(A).$

First we prove it in the case $E\cong (\ZZ/p)^n.$ Take a $p$-basic subgroup $A'$ of $A.$ Then $A'/p\cong A/p,$ the sequence ${}_pA'\mono {}_pA\epi {}_p(A/A')$ is short exact  and $A'\cong \ZZ^{\oplus k}\oplus \ZZ/p^{m_1}\oplus \dots \oplus \ZZ/p^{m_l}.$ Using the monomorphisms $A'\otimes \QQ\mono A\otimes \QQ$ and ${}_pA'\mono {}_pA,$ we obtain $k\leq {\sf dim}_\QQ(A\otimes \QQ)$ and $l\leq {\sf dim}_{\ZZ/p}({}_pA).$ Using the isomorphism $A'/p\cong A/p,$ we get ${\sf dim}_{\ZZ/p}(A/p)=k+l\leq D_p(A).$   The epimorphism $A\epi E$ induces an epimorphism $A/p\epi E.$ It follows that $D_p(E)=n\leq k+l\leq {\sf dim}_\QQ(A\otimes \QQ)+{\sf dim}_{\ZZ/p}({}_pA)=D_p(A).$

Now we prove the general case. Denote the preimage of ${}_pE$ in $A$ by $\tilde A.$ Then we have an epimorphism $\tilde A\epi {}_pE.$ As we proved above, ${\sf dim}_{\ZZ/p}({}_pE)\leq {\sf dim}_\QQ(\tilde A\otimes \QQ)+{\sf dim}_{\ZZ/p}({}_p\tilde A) .$
Consider the diagram with exact rows
$$
\begin{tikzcd} 0\arrow[r]& \tilde A \arrow[r]\arrow[d,twoheadrightarrow] & A \arrow[r]\arrow[d,twoheadrightarrow] & A/\tilde A\arrow[r] \arrow[d,twoheadrightarrow] &0\\
0 \arrow[r]& {}_pE\arrow[r] & E\arrow[r] & E/{}_pE\arrow[r] & 0
\end{tikzcd}
$$ 
Tensoring by $\QQ$ we obtain the  following.
$$
\begin{tikzcd} 0\arrow[r]& \tilde A\otimes \QQ \arrow[r] & A\otimes \QQ \arrow[r]\arrow[d,twoheadrightarrow] & (A/\tilde A)\otimes \QQ \arrow[r] \arrow[d,twoheadrightarrow] &0\\
 &  & E\otimes \QQ\arrow[r,"\cong"] & (E/{}_pE)\otimes \QQ  & 
\end{tikzcd}
$$ 
Hence, ${\sf dim}_\QQ(E\otimes \QQ)\leq {\sf dim}_\QQ((A/\tilde A)\otimes \QQ).$ Using the two obtained inequalities and the equality ${\sf dim}_\QQ(A\otimes \QQ)={\sf dim}_\QQ((A/\tilde A)\otimes \QQ)+{\sf dim}_\QQ(\tilde A\otimes \QQ) $ we get 
$${\sf dim}_\QQ(E\otimes \QQ)+{\sf dim}_{\ZZ/p}({}_pE)\leq $$
$$\leq {\sf dim}_\QQ((A/\tilde A)\otimes \QQ) +{\sf dim}_\QQ(\tilde A\otimes \QQ) +{\sf dim}_{\ZZ/p}({}_pA)={\sf dim}_\QQ(A\otimes \QQ)+{\sf dim}_{\ZZ/p}({}_pA).$$
\end{proof}

\begin{Lemma}\label{lemma_finite_prufer_homology} Let $A$ be an abelian group of finite Pr\"ufer rank, $M$ be a finite $\ZZ/p[A]$-module and $n\geq 0$. Then  
$${\sf dim}_{\ZZ/p}(H_n(A,M))\leq n^{D_p(A)-1}\cdot {\sf dim}_{\ZZ/p}(M).$$ 
\end{Lemma}
\begin{proof}
We denote by ${}_{(\text{non-}p)}A$ the  sum of $q$-power torsion subgroups for $q\ne p$  $${}_{(\text{non-}p)}A=\bigoplus_{q\ne p}\ {}_{(q)}A.$$
Any finitely generated subgroup of $A'\leq {}_{(\text{non-}p)}A$ is a finite group of order prime to $p.$ 
Then by Maschke's theorem we obtain $H_n(A',M)=0$ for $n\geq 1.$ Since homology commute with direct limits and the group is the direct limit of its finitely generated subgroups, we get  
$H_n({}_{(\text{non-}p)}A,M)=0$ for $n\geq 1.$ Using the spectral sequence of the short exact sequence 
${}_{(\text{non-}p)}A\mono A \epi A/{}_{(\text{non-}p)}A,$ 
we obtain
 $H_*(A,M)\cong H_*(A/{}_{(\text{non-}p)}A,M_{{}_{(\text{non-}p)}A}).$ Moreover, $D_p(A)=D_p(A/{}_{(\text{non-}p)}A).$ 
 Then we can assume that $A$ is $q$-torsion free for $q\ne p.$ 
 
Further in the proof we assume that $A$ is $q$-torsion free for a prime number $q$ distinct from $p$.

Prove for the case where $A$ is one of the cyclic groups $\ZZ,\ZZ/p^k.$   If $A=\ZZ,$ then $H_1(A,M)=M^A\leq M,$ $H_n(A,M)=0$ for $n\geq 2$ and $D_p(A)=1.$ If $A=\ZZ/p^k=\langle t \rangle,$ then, if we set $\mathcal N=1+t+t^2+\dots +t^{p^k-1}\in \ZZ/p[A],$ we get $H_{2m+1}(A,M)=M^A/M\cdot \mathcal N,$  $H_{2m}(A,M)=\{m\in M\mid m\cdot \mathcal N=0 \}/ M(1-t)$ and $D_p(A)=1.$ In all these cases the equation holds. 

Now we prove for the case of finitely generated $A$. Then $A=A_1\times \dots \times A_m,$ where $A_i$ is one of the groups $\ZZ,\ZZ/p^k.$  The proof is by induction on $m$. Note that $D_p(A)=m.$ The base step was proved above. Set $\tilde A=A_2\times \dots \times A_m.$ Consider the spectral sequence $H_*(A_1,H_*(\tilde A,M))\Rightarrow H_*(A,M).$ By induction hypothesis we have ${\sf dim} (H_i(\tilde A,M))\leq (n+1)^{m-2} \cdot {\sf dim}(M)$ for $i\leq n$. Then ${\sf dim} (H_{n-i}(A_1, H_i(\tilde A,M)))\leq (n+1)^{m-2} \cdot {\sf dim}(M).$ It follows that ${\sf dim}(H_n(A,M))\leq (n+1)\cdot (n+1)^{m-2}\cdot {\sf dim}(M)=(n+1)^{m-1}\cdot {\sf dim}(M).$

Prove the general case. Consider a finitely generated subgroup  $A'\leq A.$ Since $D_p(A')\leq D_p(A),$ using that we proved above, we obtain that ${\sf dim}(H_n(A',M))\leq (n+1)^{D_p(A)-1}\cdot {\sf dim}(M).$
Then the assertion follows from  Lemma \ref{lemma_direct_system_of_finit_length}, the fact that homology commute with direct filtered limits and that $A$ is the direct limit of its finitely generated subgroups.
\end{proof}

\begin{Corollary}\label{corollary_finite_prufer_homology_of_quotients}
Let $A$ be an abelian group of finite Pr\"ufer rank, $M$ be a finite $\ZZ/p[A]$-module and $n\geq 0.$ Then for any subgroup $B\leq A$ $${\sf dim}_{\mathbb Z/p}(H_n(A/B,M_B))\leq (n+1)^{D_p(A)-1}\cdot {\sf dim}_{\ZZ/p}(M).$$ In particular, for fixed $A,M,n$ the dimensions of $H_n(A/B,M_B)$ are bounded over all subgroups $B\leq A.$ 
\end{Corollary}

\begin{Corollary}\label{corollary_finite_homology} If $A$ is an abelian group such that $A/p$ and ${}_pA$ are finite for any $p$, then $H_n(A,\ZZ/p)$ is finite. 
\end{Corollary}
\begin{proof}
Take the torsion subgroup ${\sf tor}(A)=\bigoplus {}_{(p)} A$ of $A.$ Then ${}_{(p)}A$ is a an abelian group of finite Pr\"ufer rank.  By Lemma \ref{lemma_finite_prufer_homology} $H_n({\sf tor}(A),\ZZ/p)=H_n({}_{(p)}A,\ZZ/p)$ is finite for any $n$. Take the quotient ${\sf tf}(A)=A/{\sf tor}(A).$ Then $H_n({\sf tf}(A),\ZZ/p)=\Lambda^n({\sf tf}(A)/p),$ and hence, it is finite for any $n$. It follows that for any finite $\ZZ/p$-vector space $H_n({\sf tf}(A),V)$ is finite. Consider the spectral sequence of the short exact sequence ${\sf tor}(A)\mono A\epi {\sf tf}(A).$ Components of its second page $H_*({\sf tf}(A),H_*({\sf tor}(A),\ZZ/p))$ are finite. Then $H_n(A,\ZZ/p)$ is finite. 
\end{proof}

\begin{Lemma}\label{lemma_inverse_limit_of_AQp} Let $A$ be a abelian group of finite Pr\"ufer rank and $A=A_1\supseteq A_2\supseteq \dots$ a sequence of subgroups. Set $\hat A=\varprojlim A/A_i.$ Then $\hat A/p$ and ${}_p \hat A$ are finite and $$ {\sf dim}_{\ZZ/p}(\hat A/p)\leq D_p(A),\hspace{1cm} {\sf dim}_{\ZZ/p}({}_p\hat A)\leq D_p(A).$$
\end{Lemma}
\begin{proof} Set $B_i=A/A_i.$ Lemma \ref{lemma_p_limits} implies that $\hat A/p=\varprojlim (B_i/p)$ and ${}_p\hat A=\varprojlim ({}_pB_i).$ By Lemma \ref{lemma_D(A)} we obtain that dimensions of $B_i/p$ and ${}_pB_i$ are bounded by $D_p(A)$. Then Lemma \ref{lemma_direct_system_of_finit_length} implies that the sequences ${}_pB_i$ and $B_i/p$ are quasistable and the dimensions of ${}_p\hat A$ and $\hat A/p$ are bounded by $D_p(A).$ 
\end{proof}

\subsection{Completions of abelian and solvable groups of finite Pr\"ufer rank}

\begin{Proposition}\label{prop_inverse_limit_of_abelian} Let $A$ be an abelian group such that $A/p$ and ${}_pA$ are finite for any $p$  and $A=A_1\supseteq A_2\supseteq \dots$ be a sequence of subgroups and $M$ be an $\ZZ/p[A]$-module. Assume that
\begin{enumerate}
\item the groups ${}_{p}(A/A_i)$ and $ A_i/p$ are finite for any $i;$

\item for any $i$ there is $j>i$ such that $A_j\subseteq pA_i;$

\item ${}_pA_i=0$ for big enough $i;$

\item $A_i$ acts trivially on $M$ for big enough $i.$
\end{enumerate}
Then the obvious morphisms are isomorphisms  $$H_*( A,M)\cong \varprojlim H_*(A/A_i,M_{A_i}), \hspace{1cm} H^*( A,M)\cong \varinjlim H^*(A/A_i,M^{A_i})
.$$ 
\end{Proposition}
\begin{proof} Since  $A_i$ acts trivially on $M$ for big enough $i,$ we have $H_n(A_i,M)\cong H_n(A_i,\ZZ/p)\otimes M$ for big enough $i.$
 By assumption for any fixed $i$ there exist $j>i$ such that 
 $A_j\subseteq pA_i$ and ${}_pA_j=0.$ Then by Lemma 
\ref{lemma_f_induces_homology} we obtain that for any $n$ and big enough $i$ (so big that $M$ is a trivial $A_i$-module) there exists $j>i$ such that the map $H_n(A_j,M)\to H_n(A_i,M)$ is trivial. Therefore, the sequence $H_n(A_i,M)$ is zero-equivalent for any $n.$ For any $i$ we have a spectral sequence $H_*(A/A_i,H_*(A_i,M))\Rightarrow H_*(A,M).$  Since ${}_p(A/A_i), A_i/p$ are finite, we have that the second page consists of finite groups (Corollary \ref{corollary_finite_homology}). 
The functor of inverse limit is exact on the category of finite abilian groups, and hence, we get the spectral sequence $\varprojlim H_*(A/A_i,H_*(A_i,M)) \Rightarrow H_*(A,M).$ 
Since the inverse sequence  $H_n(A_i,M)$ is zero-equivalent, we obtain $\varprojlim H_*(A/A_i, H_n(A_i,M))=0$ for $n\geq 1.$ 
Then $\varprojlim H_*(A/A_i,M)\cong H_*(A,M).$ 

For cohomology the proof is similar, if we use the formula $H^*(A,M)\cong H_*(A,M^\vee)^\vee$ for a group $A$ and a finite $\ZZ/p[A]$-module $M$. 
\end{proof}

\begin{Proposition}\label{proposition_finite_prufer_abelian} Let $A$ be an abelian group of finite Pr\"ufer rank and let $A\supseteq A_1\supseteq A_2\supseteq \dots$ be a sequence of subgroups.  Set $\hat A=\varprojlim A/A_i$ and $\hat A_i = {\sf Ker}(\hat A\to A/A_i).$ Assume that $M$ is a finite  $\ZZ/p[\hat A]$-module. Then $\hat A_i$ acts trivially on $M$ for big enough $i,$ the dimensions of the sequences $H_n(A/A_i,M_{\hat A_i})$ and $H^n(A/A_i,M^{\hat A_i})$ are bonded for any $n,$ the groups $H_n(\hat A, M)$ and $H^n(\hat A, M) $ are finite for any $n,$ and the obvious maps  are isomorphisms $$ H_*(\hat A,  M )\cong \varprojlim H_*(A/A_i,M_{\hat A_i}), $$
$$ H^*(\hat A,  M)\cong \varinjlim H^*(A/A_i,M^{\hat A_i})\cong H^*_u(\hat A, M).$$
In particular, $M$ is a cohomologically discrete $\hat A$-module. 
\end{Proposition}
\begin{proof} 
 It is easy to see that $\hat A_k={\varprojlim}_{i>k} A_k/A_i,$ that $\hat A/\hat A_k\cong A/A_k$ are of finite Pr\"ufer rank and $\bigcap \hat A_i=0.$ In particular, ${}_p\hat A_i=0$ for big enough $i.$ By Lemma \ref{lemma_inverse_limit_of_AQp} we obtain that $\hat A_k/p$ and ${}_p(\hat A_i)$ are finite. By Lemma \ref{lemma_commplete_and_finite} we obtain that for any $i$ there is $j>i$ such that $\hat A_j\subseteq p\hat A_i.$ Since $M$ is finite, $\mathsf G={\sf Aut}_\ZZ(M)$ is finite. The action of $\hat A$ on $M$ is given by a homomorphism $ \hat A\to \mathsf G.$ Since $ |\mathsf G| \cdot \hat A $ is in the kernel of the homomorphism, $ \hat A_i$ acts trivially on $M$ for big $i$ (Lemma \ref{lemma_commplete_and_finite}).
 Then we can use Proposition \ref{prop_inverse_limit_of_abelian} for the sequence  $\hat A \supseteq \hat A_1 \supseteq \hat A_2 \dots$  and obtain the isomorphisms $H_*(\hat A,  M )\cong \varprojlim H_*(A/A_i,M_{\hat A_i}) $ and $H^*(\hat A,  M)\cong \varinjlim H^*(A/A_i,M^{\hat A_i})$. Moreover, we obtain $M=M_{\hat A_i}=M^{\hat A_i}$ for big enough $i.$ By Corollary \ref{corollary_finite_prufer_homology_of_quotients} we obtain that the dimensions of $H_n(A/A_i,M_{\hat  A_i})$ and $H^n(A/A_i,M^{\hat A_i})$ are bounded for a fixed $n$. Then Lemma \ref{lemma_direct_system_of_finit_length} implies that these sequences are quasiconstant.  The last isomorphism follows from Proposition \ref{proposition_uniform_cohomology}.
\end{proof}

Note that  solvable group $G$ is of finite Pr\"ufer rank if and only if  there is a finite sequence of normal subgroups $G=U_1\supseteq U_2 \supseteq \dots \supseteq U_s=1$ such that $U_i/U_{i+1}$ are abelian groups of finite Pr\"ufer rank.

\begin{Theorem}\label{theorem_finite_prufer_solvable} Let $G$ be a  solvable group group of finite Pr\"ufer rank and let $G\supseteq G_1\supseteq G_2\supseteq \dots$ be a sequence of normal subgroups.  Set $\hat G=\varprojlim G/G_i$ and $\hat G_i = {\sf Ker}(\hat G\to G/G_i).$ Assume that $M$ is a finite  $\ZZ/p[\hat G]$-module. Then $\hat G_i$ acts trivially on $M$ for big enough $i,$ the dimensions of the sequences $H_n(G/G_i,M_{\hat G_i})$ and $H^n(G/G_i,M^{\hat G_i})$ are bounded for any $n,$ the groups $H_n(\hat G, M)$ and $H^n(\hat G, M) $ are finite for any $n,$ and the obvious maps  are isomorphisms $$H_*(\hat G,  M )\cong \varprojlim H_*(G/G_i,M_{\hat G_i}), $$
$$H^*(\hat G,  M)\cong \varinjlim H^*(G/G_i,M^{\hat A_i})\cong H^*_u(\hat G, M).$$
In particular, $M$ is a cohomologically discrete $\hat G$-module. 
\end{Theorem}
\begin{proof} 
Let $G= U_1\supseteq U_2 \supseteq \dots \supseteq U_s=1$ be a finite sequence of normal subgroups such that $U_i/U_{i+1}$ are abelian groups of finite Pr\"ufer rank and $M$ be a finite $\ZZ/p[\hat G]$-module. The proof is by induction on $s.$ If $s=1,$ it follows from Proposition  \ref{proposition_finite_prufer_abelian}. Prove the induction step. Set $A=U_{s-1},$  $\ A_i=A\cap G_i,$ $\ \hat A=\varprojlim A/A_i,$ $\ G'=G/A,$ $\ G'_i=G_iA/A,$  $\ \hat G'=\varprojlim G'/G'_i.$ Then there are short exact sequences $\hat A\mono \hat G \epi \hat G'.$ Note that $G'/G'_i=G/G_iA.$ It follows that there is a short exact sequence $A/A_i \mono G/G_i \epi G'/G'_i.$ We set $\hat A_i={\rm Ker}(\hat A\to A/A_i),$ $\hat G_i={\rm Ker}(\hat G\to G/G_i)$ and $\hat G'_i={\rm Ker}(\hat G'\to G'/G'_i).$ Then we obtain a commutative diagram, whose columns are short exact sequences and the second and third rows are short exact sequences.  
$$
\begin{tikzcd} 
\hat A_i \arrow[r] \arrow[d]& \hat G_i \arrow[r] \arrow[d] & \hat G'_i  \arrow[d]  \\
\hat A \arrow[r] \arrow[d]& \hat G \arrow[r] \arrow[d] & \hat G' \arrow[d] \\
  A/A_i \arrow[r] & G/G_i \arrow[r] &  G'/G'_i.
\end{tikzcd}
$$
Using $3\times 3$-lemma, we obtain that $\hat A_i\mono \hat G_i \epi \hat G'_i$ is a short exact sequence. By Proposition \ref{proposition_finite_prufer_abelian} $\hat A_i$ acts trivially on $M$ for big enough $i.$ It follows that the action of $\hat G_i$ on $M$ factors through $\hat G'_i$ for big enough $i.$ Then using  the induction hypothesis, we obtain $\hat G'_j$ acts trivially on $M$ for big enough $j.$ Therefore, $\hat G_j$ acts trivially on $M$ for big enough $j.$ 
By Proposition \ref{proposition_finite_prufer_abelian} we obtain that there is an isomorphism $H^*(\hat A,M)\cong \varinjlim H^*(A/A_i,M^{\hat A_i})$, all these groups are finite $ \mathbb Z/p[\hat G']$-modules, where $\hat G'=\varprojlim G/G_iA,$ and the sequences $H^*(A/A_i,M^{\hat A_i})$ have bounded dimensions, and hence, they are quasiconstant.    Then by induction hypothesis we have $H^*(\hat G',H^*(\hat A,M))\cong H_u^*(\hat G',H^*(\hat A,M)).$ By Proposition \ref{proposition_uniform_cohomology} we have $H_u^*(\hat G',H^*(\hat A,M))\cong \varinjlim H^*(G'/G_i', H^*(A/A_i,M)).$ It follows that $H^*(\hat G,M)\cong \varinjlim H^*(G/G_i,M^{\hat G_i}).$ Proposition \ref{proposition_uniform_cohomology} implies that $ \varinjlim H^*(G/G_i,M^{\hat G_i})\cong H^*_u(\hat G,M).$ The proof for homology is similar. 
\end{proof}

\section{ $R$-completions }\label{sectionRcompletions}

\subsection{Definitions}

Let $R\in \{\ZZ,\ZZ/p,\QQ\}.$ Following Bousfield \cite{Bousfield77} we define  lower $R$-central series of a group $G$ by recursion 
$$\gamma_{i+1}^R G = {\rm Ker}(\gamma_i^R G\to \gamma_i^R G/[\gamma_i^R G,G] \otimes R),$$
where $\gamma_1^RG=G.$ Then $R$-completion of $G$ is defined as follows
$$\hat G_R=\begin{cases}\ \varprojlim G/\gamma_i^RG, &\text{ if } R\in \{\ZZ,\ZZ/p\}\\
\ \varprojlim (G/\gamma_i^\QQ G \otimes \QQ), & \text{ if } R=\QQ,
  \end{cases}$$ 
 where $-\otimes \QQ$ is Malcev completion functor. 
 
\begin{Remark} It is easy to prove by induction that $$\gamma_i^\QQ G=\sqrt{\gamma_i^\ZZ G}:=\{g\in G\mid g^n\in \gamma_i^\ZZ G \text{ for some } n \}.$$ It follows that $G/\gamma_i^\QQ G \otimes \QQ = G/\gamma_i^\ZZ G \otimes \QQ.$ Then  $\hat G_\QQ=\varprojlim (G/\gamma_i^\ZZ G \otimes \QQ).$ 
\end{Remark}

A group $G$ is said to be $HR$-good if the map $H_2(G,R)\to H_2(\hat G_R,R)$ is an epimorphism. Moreover, it is interesting when the map $H_2(G,\ZZ/p) \to H_2(\hat G_\ZZ,\ZZ/p)$ is an epimorphism. 

$G$ is called $R$-prenilpotent if $\gamma_i^RG=\gamma_{i+1}^RG$ for big enough $i.$ In this case $\hat G_R=G/\gamma_i^RG$ for $R\in\{\ZZ,\ZZ/p\}$ and $\hat G_\QQ=G/\gamma_i^\QQ \otimes \QQ$ for big enough $i.$ It follows that $\hat G_R$ is discrete, and in particular, cohomomologically discrete over any ring.

\subsection{$2$-cohomologicaly discrete  $R$-completions}

We are interested in the map 
\begin{equation}\label{eq_main_map_homol}
H_n(G,K)\longrightarrow H_n(\hat G_R,K).
\end{equation}
We denote by $\Phi_\alpha^RH_n(G,K)$  $\alpha$th term of Dwyer filtration, which is defined as  $$\Phi_\alpha^RH_n(G,K)={\rm Ker}(H_n(G,K)\to H_n(G/\gamma^R_\alpha G,K)).$$  Using that $H_n(G/\gamma^\QQ_iG,\QQ)=H_n(G/\gamma^\QQ_iG \otimes \QQ,\QQ)$ (Proposition 2.5 of \cite{Hilton}) we obtain that the map  \eqref{eq_main_map_homol} factors through $H_n(G/\gamma^R_iG,K)$ for any finite $i.$ On the other hand, the map $G\to \hat G_R$ factors through $G/\gamma_\omega G.$ Hence
\begin{equation}
 \Phi_\omega^RH_n(G,K)\subseteq  { \rm Ker}(H_n(G,K)\to H_n(\hat G_R,K))\subseteq \ \bigcap_{i=1}^\infty\: \Phi_i^RH_n(G,K).
\end{equation}

\begin{Lemma}\label{lemma_short_exact_seq_lim_hom} Let $G$ be a group and $S$ be a quotient of $R.$ 
Set $\Phi_i:=\Phi_i^RH_2(G,S).$   
\begin{enumerate}
\item Then $\varprojlim H_2(G,S)/\Phi_i=\varprojlim H_2(G/\gamma_i^RG,S)$ 
and there is an exact sequence 
$$0 \to \bigcap_{i=1}^\infty \Phi_i \to H_2(G,S) \to \varprojlim H_2(G/\gamma_i^RG,S) \to  {\varprojlim}^1 \Phi_i \to 0 .$$ 
\item Moreover, if $S=K$ and $H_2(G,K)$ is finite dimensional, then  there are short exact sequences
$$0 \longrightarrow \Phi_j \longrightarrow H_2(G,K) \longrightarrow \varprojlim H_2(G/\gamma_i^RG,K) \longrightarrow 0$$
for  big enough $j,$ and the inverse sequence $H_2(G/\gamma_i^RG,K)$ is quasiconstant. 
\end{enumerate}
\end{Lemma}
\begin{proof} The 5-term exact sequence gives the exact sequence
\begin{equation}\label{eq_exact_nil_gamma}
0 \longrightarrow H_2(G,S)/\Phi_i \longrightarrow H_2(G/\gamma_i^RG,S) \longrightarrow \gamma_i^RG/\gamma_i^RG\otimes S \longrightarrow 0.
\end{equation}
Since $\gamma_i^RG/\gamma_i^RG\otimes S $ is zero-equivalent, $\varprojlim H_2(G,S)/\Phi_i=\varprojlim H_2(G/\gamma_i^RG,S).$ Then the (1) follows from the short exact sequence $\Phi_i\mono H_2(G,S)\epi H_2(G,S)/\Phi_i.$ 

If $S=K$ and $H_2(G,K)$ is finite dimensional, then $V_i$ and $\Phi_i$ stabilises. Hence $\varprojlim{}^1\Phi_i=0$ and \eqref{eq_exact_nil_gamma} implies that $H_2(G/\gamma_i^RG,K)$ is quasiconstant. 
\end{proof}

\begin{Proposition}\label{prop_eqviv_2-coh-discr} Let $G$ be a group such that $H_1(G,R)$ is finitely generated over $R$ and $H_2(G,K)$ is finite dimensional. Set $\Phi_i:=\Phi^R_iH_2(G,K).$ Then the following are equivalent. 
\begin{enumerate}
\item $\hat G_R$ is $2$-cohomologically discrete over $K$.

\item The obvious maps give a short exact sequence
$$0 \longrightarrow \Phi_i \longrightarrow H_2(G,K)\longrightarrow H_2(\hat G_R,K) \longrightarrow 0$$
for  big enough $i.$
\item $H_2(\hat G_R,K)\cong \varprojlim H_2(G/\gamma_i^RG,K).$
\end{enumerate}
\end{Proposition}
\begin{proof} Since $H_1(G,R)$ is finitely generated over $R,$ we have $H_1(G,R)\cong H_1(\hat G_R,R)$ (see \cite{Bousfield77}). Thus $H^1(G,K)=H^1(G/\gamma_i^RG,K)=H^1(\hat G_R,K)$ for any $i>1.$ Hence we need to think only about the second (co)homology. 

 Lemma \ref{lemma_short_exact_seq_lim_hom} implies that $(2)$ and $(3)$ are equivalent.  Moreover, it implies that $H_2(G/\gamma_i^RG,K)$ is quasiconstant and its inverse limit  is finite dimensional. By Lemma \ref{lemma_contravarian_functor} we get that $ H^2(G/\gamma_i^RG,K)$ in a quasiconstant direct sequence with a finite dimensional inverse limit.  It follows that $(\varprojlim H_2(G/\gamma^R_iG,K))^*=\varinjlim H^2(G/\gamma^R_iG,K)\cong H_u^2(\hat G_R,K)$ and $(\varinjlim H^2(G/\gamma^R_iG,K))^*=\varprojlim H_2(G/\gamma^R_iG,K),$ where $(-)^*={\rm Hom}_K(-,K).$   Using this, we get that $(3)$ is equivalent to $(1).$
\end{proof}

\begin{Remark} Using Proposition \ref{prop_eqviv_2-coh-discr}, the main result of \cite{Ivanov_Mikhailov} can be reformulated as follows.  If $G$ is a finitely presented metabelian group, then $\hat G_R$ is $2$-cohomologically discrete over $K$.
\end{Remark}

\subsection{$\QQ$-prenilpotent groups}

Note that by definition $G/\gamma_i^\QQ G$ is torsion free.

\begin{Lemma}\label{lemma_prenilpotent_stable_sequence} A group $G$ is $\QQ$-prenilpotent if and only if the sequence $G/\gamma^\QQ_iG \otimes \QQ$ stabilises. 
\end{Lemma}
\begin{proof}
If $G$ is $\QQ$-prenilpotent, then obviously the sequence $G/\gamma^\QQ_iG \otimes \QQ$ stabilises.
Assume now that  $G/\gamma^\QQ_iG \otimes \QQ$ stabilises and prove that $G$ is $\QQ$-prenilpotent. Since $-\otimes \QQ$ is an exact functor, this means that $\gamma_i^\QQ G/\gamma_{i+1}^\QQ G \otimes \QQ=0$ for  big enough $i.$  Using that $\gamma_i^\QQ G/\gamma_{i+1}^\QQ G$ is torsion free, we obtain that $\gamma_i^\QQ G/\gamma_{i+1}^\QQ G=0$ for  big enough $i.$ Thus $G$ is $\QQ$-prenilpotent. 
\end{proof}

\begin{Lemma} \label{lemma_extension_of_Q-prenilpotent}
Let $A\mono G\epi G''$ be a short exact sequence of groups such that 
\begin{enumerate}
\item $G''$ is $\QQ$-prenilpotent.
\item $A$ is an abelian group such that $A\otimes \QQ$ is finite dimensional. 
\end{enumerate}  
Then $G$ is $\QQ$-prenilpotent. 
\end{Lemma}
\begin{proof} We assue that $A$ is a normal subgroup in $G$ and $A\mono G$ is the embedding. 
Set $V_i=(A/A\cap \gamma^\QQ_iG) \otimes \QQ.$  Then $V_i$ is a quotient of $A\otimes \QQ.$  Consider the short exact sequence
$$0\longrightarrow  V_i \longrightarrow G/\gamma^\QQ_iG \otimes \QQ  \longrightarrow G''/\gamma^\QQ_i G'' \otimes \QQ \longrightarrow 1.$$ Using that $G''$ is 
$\QQ$-prenilpotent and $A\otimes \QQ$ is finitedimensional, we get that $G''/\gamma^\QQ_i G'' \otimes \QQ$ and $V_i$ stabilise. Then $G/\gamma^\QQ_iG \otimes \QQ$ stabilises. Then by Lemma \ref{lemma_prenilpotent_stable_sequence} we obtain that $G$ is $\QQ$-prenilpotent.  
\end{proof}

\begin{Proposition}\label{proposition_solvable_Q-prenilpotent} Let $G$ be a solvable group with a finite sequence of normal subgroups $G=U_1\supseteq U_2 \supseteq \dots \supseteq U_s=1$ such that $U_i/U_{i+1}$ are abelian and $U_i/U_{i+1}\otimes \QQ$ are finite dimensional. Then $G$ is $\QQ$-prenilpotent. In particular, $\hat G_\QQ$ is cohomologically discrete over $\QQ.$  Moreover, if $G$ is finitely generated, then it is $H\QQ$-good.
\end{Proposition}
\begin{proof} Using the Lyndon–Hochschild–Serre spectral sequence and induction on $s$, it is easy to see that $H_2(G,\QQ)$ is finite dimensional.  Then the assertion follows from Lemma \ref{lemma_extension_of_Q-prenilpotent} by induction and Proposition \ref{prop_eqviv_2-coh-discr}. 
\end{proof}

\begin{Remark} Note that we can not replace $\QQ$ by $\ZZ$ or $\ZZ/p$ in Proposition \ref{proposition_solvable_Q-prenilpotent}. For example, consider the group $G=\ZZ \rtimes \ZZ/2$ with the sign action of $\ZZ/2$ on $\ZZ.$ Then we have a short exact sequence $\ZZ\mono G \epi \ZZ/2,$ where $\ZZ/2\otimes R$ and $\ZZ\otimes R$ are finitely generated over $R$ for $R=\ZZ,\ZZ/2.$ But $G$ is not $R$-prenilpotent for $R=\ZZ,\ZZ/2$ because 
$\gamma_i^\ZZ(G)=\gamma^{\ZZ/2}_i(G)=2^{i-1}\ZZ\rtimes 0$ for $i\geq 2.$  
\end{Remark}

\subsection{Completions of solvable groups of finite Pr\"ufer rank}

Recall that a solvable group $G$ is of finite Pr\"ufer rank if there is a finite sequence of normal subgroups $G=U_1 \supseteq \dots \supseteq U_s=1$ such that for any $i$ the group $U_i/U_{i+1}$ is an abelian group of finite Pr\"ufer rank.

\begin{Theorem}\label{theorem_solvable_R-compl} Let $G$ be a solvable group of finite Pr\"ufer rank. Then $\hat G_\ZZ$ and $\hat G_{\ZZ/p}$ are cohomologically discrete over $\ZZ/p.$ Moreover, if $G$ is finitely generated, then $H_2(G,\ZZ/p)\to H_2(\hat G_R,\ZZ/p)$ is an epimorphism for $R\in \{\ZZ,\ZZ/p\}$. 
\end{Theorem}
\begin{proof}
Cohomological discreteness follows from Theorem \ref{theorem_finite_prufer_solvable}. Using that $H_n(A,\ZZ/p)$ is finite for an abelian group $A$ of finite Pr\"ufer rank, by induction, it is easy to check that $H_n(G,\ZZ/p)$ is finite. Then the epimorphism follows from Proposition \ref{prop_eqviv_2-coh-discr}.
\end{proof}

\begin{Proposition}\label{prop_pro-p-abelian} Let $A$ be a abelian group such that $A/p$  and ${}_pA$ are finite. Then $\hat A_{\ZZ/p}$ is a finitely generated $\ZZ_p$-module. Moreover, if ${\sf tor}(A)$ is finite and $M$ is a finite $\ZZ/p[\hat A_{\ZZ/p}]$-module, then
\begin{equation}\label{eq_isom_quasismal_compl}
H_*(A,M)=H_*(\hat A_{\ZZ/p},M).
\end{equation}
\end{Proposition}
\begin{proof} 
The multiplication by $p^i$ gives an epimorphism $A/p\epi p^iA/p^{i+1}A.$ It follows that $p^iA/p^{i+1}A$ is finite and $A/p^i$ is finite for any $i.$ 
Then $\hat A_{\ZZ/p}$ is a profinite abelian group. Note that $A/p^i$ is a product of cyclic $p$-groups, and the number of the factors does not depend on $i$ because $(A/p^i)/p=A/p.$ It follows that $\hat A_{\ZZ/p}$ has a finite set of generators converging to $1$ (see Lemma 2.5.3 of \cite{Ribes-Zalesskii}). For any other prime number $q\ne p$ we have $(\hat A_{\ZZ/p})/q=\varprojlim (A/p^i)/q=0$ (Lemma  \ref{lemma_p_limits}). Then by Theorem 4.3.5 of \cite{Ribes-Zalesskii} we obtain $\hat A_{\ZZ/p}$ is a finitely generated $\ZZ_p$-module.

Assume that ${\sf tor}(A)$ is finite. Hence ${}_p(p^iA)=0$ for big enough $i.$ 
Note that $p^iA/p^{i+1}A$ is finite and $A/p^i$ is finite.  
Consider the group of automorphisms $G={\sf Aut}_\ZZ(M)$ and assume that $p^t$ is the largest power of $p$ that divides $|G|.$ Hence $p^t \hat A_{\ZZ/p} \subseteq {\rm Ker}(\hat A_{\ZZ/p} \to G).$  It follows that $p^t\hat A_{\ZZ/p}$ acts trivially on $M.$ Then we can use  Proposition \ref{prop_inverse_limit_of_abelian} both for $A$ and $\hat A_{\ZZ/p}$. The assertion follows. 
\end{proof}

\begin{Remark} If $A$ is an abelian group of finite Pr\"ufer rank but the torsion subgroup is not finite, the isomorphism \eqref{eq_isom_quasismal_compl} fails. For example, if  $\ZZ/p^{\infty}=\varinjlim \ZZ/p^i,$ then $(\ZZ/p^{\infty})^{\wedge}_{\ZZ/p}=0 $ and
$$\ZZ/p=H_2(\ZZ/p^\infty,\ZZ/p) \ne H_2((\ZZ/p^{\infty})^{\wedge}_{\ZZ/p},\ZZ/p)=0.$$
\end{Remark}

\section{Metabelian groups over $C$}\label{sectionmetabliangroupsoverC}

\subsection{Completion of modules}

Let $\Lambda$ 
be a Noetherian commutative ring and 
$\mathfrak{a} \triangleleft \Lambda$ be an ideal. The 
$\mathfrak{a}$-adic completion of a 
$\Lambda$-module $M$ is the inverse limit 
$\hat M_{\mathfrak{a}}=\varprojlim M/M\mathfrak{a}^i.$ Denote by 
$\varphi_M:M\to \hat M_{\mathfrak{a}}$ the obvious map. 
The completion of the ring itself $\hat \Lambda_{\mathfrak{a}}$ is a Noetherian commutative ring \cite[III, \S 3, Cor. of Prop.8]{Bourbaki}  and 
$\hat M_{\mathfrak{a}}$ has a natural structure of $\hat \Lambda_{\mathfrak{a}}$-module. Moreover,
\begin{equation}\label{eq_hat_m}
\hat M_{\mathfrak{a}}=\varphi_M(M)\cdot \hat\Lambda_{\mathfrak{a}}
\end{equation}
\cite[III, \S 3, Cor.1]{Bourbaki}. The functor $M\mapsto \hat M_{\mathfrak{a}}$ is exact on the category of finitely generated $\Lambda$-modules, there is an isomorphism  
\begin{equation}\label{eq_hat_tensor}
\hat M_{\mathfrak{a}}\cong M\otimes_{\Lambda}\hat \Lambda_{\mathfrak{a}}
\end{equation}
for a finitely generated module $M$
 \cite[III, \S 3, Th.3]{Bourbaki}. In particular, if $N\leq M,$ we can identify $\hat N_{\mathfrak{a}}$ with the submodule in  $\hat M_{\mathfrak{a}}.$

Consider the ring homomorphism $\varphi=\varphi_\Lambda:\Lambda\to \hat \Lambda_\mathfrak{a}.$ Denote by 
$\hat{\mathfrak{a}}={\rm Ker}(\hat \Lambda_{\mathfrak{a}} \epi \Lambda/\mathfrak{a}).$ Applying the functor of the $\mathfrak{a}$-adic completion to the short exact sequence 
$\mathfrak{a}\mono \Lambda\epi \Lambda/\mathfrak{a}$ we get
$\hat{\mathfrak{a}}=\hat{\mathfrak{a}}_{\mathfrak{a}}.$  Using \eqref{eq_hat_m} we obtain $\hat{\mathfrak{a}}= \hat \Lambda_{\mathfrak{a}} \cdot \varphi(\mathfrak{a}).$ If we identify $\widehat{(\mathfrak{a}^i)}_{\mathfrak{a}}$ with the ideal of $\hat \Lambda_{\mathfrak{a}}$ and use \eqref{eq_hat_m}, we obtain
\begin{equation}\label{eq_hat_a}
\widehat{(\mathfrak{a}^i)}_{\mathfrak{a}}=\varphi_{\mathfrak{a}^i}(\mathfrak{a}^i)\cdot \hat \Lambda_{\mathfrak{a}}=\hat \Lambda_{\mathfrak{a}} \cdot \varphi(\mathfrak{a}^i)= \hat{\mathfrak{a}}^i.
\end{equation}
Tensoring the short exact sequence $\hat{ \mathfrak{a}}^i\mono \hat \Lambda_{\mathfrak{a}}\epi \Lambda/\mathfrak{a}^i$ 
on $M,$ we obtain the exact sequence $ M\otimes_{\Lambda} \hat{\mathfrak{a}}^i\to \hat M\to M/M\mathfrak{a}^i\to 0$, which implies
$$M/M \mathfrak{a}^i=\hat M_{\mathfrak{a}}/\hat M_{\mathfrak{a}}\hat{\mathfrak{a}}^i.$$
In particular, 
$\Lambda/\mathfrak{a}^i=\hat \Lambda_{\mathfrak{a}}/\hat{\mathfrak{a}}^i.$ It follows that $(\hat \Lambda_\mathfrak{a})^\wedge_{\hat{\mathfrak{a}}}=\hat \Lambda_{\mathfrak{a}}.$ Assume that $\mathcal M$ is a finitely generated $\hat \Lambda_{\mathfrak{a}}$-module. Using \eqref{eq_hat_a} we obtain $\mathcal M\cdot \hat{\mathfrak{a}}^i=\mathcal M \cdot \mathfrak{a}^i.$ Then, using \eqref{eq_hat_tensor}, the isomorphism $(\hat \Lambda_\mathfrak{a})^\wedge_{\hat{\mathfrak{a}}}=\hat \Lambda_{\mathfrak{a}}$ and the equality $\mathcal M\cdot \hat{\mathfrak{a}}^i=\mathcal M \cdot \mathfrak{a}^i,$ we get
\begin{equation}\label{eq_hat_mathcal_M}
\hat{\mathcal M}_{\hat{\mathfrak{a}}}=\hat{\mathcal M}_{\mathfrak{a}}=\mathcal M.
\end{equation}

\begin{Lemma}
Let $\Lambda$ be a Noetherian commutative ring, $\mathfrak{a},\mathfrak{b}$ are ideals of $\Lambda$   and $M$ be a finitely generated $\Lambda$-module. Then
$$(\hat M_{\mathfrak{a}})^\wedge_{\mathfrak{b}}=\hat M_{\mathfrak{a}+\mathfrak{b} }= (\hat M_{\mathfrak{b}})^\wedge_{\mathfrak{a}}.$$ 
\end{Lemma}
\begin{proof} Note that $\hat M\mathfrak{a}^i=\hat M\hat{\mathfrak{a}}^i.$ Then we have  $\hat M_{\mathfrak{a}}/\hat M_{\mathfrak{a}} \mathfrak{a}^i=M/M\mathfrak{a}^i.$ It follows that 
$$\frac{M}{M\mathfrak{a}^i+M\mathfrak{b}^j}=\frac{\hat M_{\mathfrak{a}}}{\hat M_\mathfrak{a} \mathfrak{a}^i+\hat M_\mathfrak{a} \mathfrak{b}^j}.$$ If we take $\mathcal M=\hat M_{\mathfrak{a}}/\hat M_{\mathfrak{a}}\mathfrak{b}^j,$ then \eqref{eq_hat_mathcal_M} implies $$\varprojlim\limits_{i} \ \frac{\hat M_{\mathfrak{a}}}{\hat M_\mathfrak{a} \mathfrak{a}^i+\hat M_\mathfrak{a} \mathfrak{b}^j} = \hat M_{\mathfrak{a}}/\hat M_{\mathfrak{a}}\mathfrak{b}^j.$$
Therefore we have 
$$ \varprojlim_{j}\ \varprojlim\limits_{i} \ \frac{M}{M\mathfrak{a}^i+M\mathfrak{b}^j} = (\hat M_{\mathfrak{a}})^\wedge_{\mathfrak{b}}.$$ Using that a double inverse limit is the limit over the ordered set $\mathbb N^2$ and that the diagonal $\{(i,i)\mid i\in \mathbb N\} $ is a cofinal subset, we obtain 
$$\varprojlim\limits_{i} \ \frac{M}{M\mathfrak{a}^i+M\mathfrak{b}^i} = (\hat M_{\mathfrak{a}})^\wedge_{\mathfrak{b}}.$$ Finally, using that $ M(\mathfrak{a}+\mathfrak{b})^{2i} \subseteq M\mathfrak{a}^i+M\mathfrak{b}^i\subseteq M(\mathfrak{a}+\mathfrak{b})^i,$ we obtain $\hat M_{\mathfrak{a}+\mathfrak{b}}=(\hat M_{\mathfrak{a}})^\wedge_{\mathfrak{b}}.$
\end{proof}

Let $A$ be a finitely generated abelian group and $\Lambda$ be a Noetherian commutative ring. Then the group ring $\Lambda[A]$ is a Noetherian commutative ring. We denote by $I_\Lambda$ the augmentation ideal of $\Lambda[A]$, denote by $p\cdot \Lambda[A]$ the ideal of $\ZZ[A]$ generated by $p\in \Lambda\subset \Lambda[A]$ and set $$I_{p,\Lambda}:=I_{\Lambda}+p\cdot \Lambda[A]= {\rm Ker}(\Lambda[A] \epi \Lambda/p).$$
Moreover, we set 
$$I=I_\ZZ, \hspace{1cm} I_p=I_{p,\ZZ}.$$

 It is easy to see that, if $M$ is a $\Lambda[A]$-module, then its $\ZZ/p$-completion as an abelian group coincides with the $p\cdot \Lambda[A]$-adic completion  $$\hat M_{\ZZ/p}=\hat M_{p\cdot \Lambda[A]}.$$
In particular, it does not depend on $\Lambda.$ Moreover, if $A$ is generated by $t_1,\dots,t_n,$
 then $I_\Lambda=\sum \Lambda[A](t_i-1)$   and $$MI_\Lambda=\sum M(t_i-1)=MI, \hspace{1cm} MI_{p,\Lambda}=MI+pM=MI_p.$$ It follows that the completions do not depend on $\Lambda$ 
$$\hat M_{I_\Lambda}=\hat M_I,\hspace{1cm} \hat M_{I_{p,\Lambda}}=\hat M_{I_p} $$
and we can use the notations $\hat M_I$ and $\hat M_{I_p}$ for any $\Lambda.$

\begin{Corollary}\label{cor_comm_compl} Let $A$ be a finitely generated abelian group, $\Lambda$ be a Noetherian commutative ring and $M$ be a finitely generated $\Lambda[A]$-module. Then 
$$(\hat M_{\ZZ/p})^\wedge_{I}=\hat M_{I_p}=(\hat M_{I})^\wedge_{\ZZ/p}.$$ 
\end{Corollary}

By $\ZZ_p$ we denote the ring of $p$-adic integers 
$\ZZ_p=\varprojlim \ZZ/p^i =\hat \ZZ_{\ZZ/p}.$

\begin{Corollary}\label{cor_comm_Z_p-mod} Let $A$ be a finitely generated abelian group and $M$ be a $\ZZ_p[A]$-module, which is finitely generated over $\ZZ_p.$ Then 
$$ \varprojlim\ (M/p^i)^\wedge_I=\hat M_I = \hat M_{I_p}.$$
\end{Corollary}
\begin{proof}
It follows from the equation $(M/p^i)^\wedge_I=(\ZZ/p^i\otimes M) \otimes_{\ZZ_p[A]} \widehat{\ZZ_p[A]}_I  =\ZZ/p^i\otimes (M \otimes_{\ZZ_p[A]} \widehat{\ZZ_p[A]}_I)= (\hat M_I)/p^i,$ the equation $\hat M_{\ZZ/p}=M$ and Corollary  \ref{cor_comm_compl}.
\end{proof}

By $C$ we denote the infinite cyclic group.

\begin{Lemma}\label{lemma_k[c]-mod} Let $ \mathcal K$ be an Artinian commutative ring and $M$ be a $\mathcal K[C]$-module, which is finitely generated over $\mathcal K.$ Then the $C$-module $\hat M_I$ is nilpotent, the homomorphism $M\to \hat M_I$ is a split homomorphism of $C$-modules and there is an isomorphism $$M=MI^{\infty} \oplus \hat M_I,$$
which is natural by $M,$ 
 where $MI^\infty=\bigcap_{i=1}^\infty MI^i.$ 
\end{Lemma}
\begin{proof}
We denote by $a:M\to M$ the multiplication on $t$ and set $b:=a-1.$  Then there is positive integer $n$ such that ${\rm Im}(b^n)={\rm Im}(b^{n+1}),$ ${\rm Ker}(b^n)={\rm Ker}(b^{n+1})$ and $V={\rm Im}(b^n) \oplus {\rm Ker}(b^n)$ \cite[Cor. 6.4.2]{Kash}. Moreover, it is easy to see that $MI^i={\rm Im}(b^i).$ It follows that $MI^{i}={\rm Im}(b^n)$ for $i\geq n,$  $\hat V_{I}={\rm Ker}(b^n)$ and $V\cong \hat V_{I}\oplus VI^i.$
\end{proof}

\begin{Lemma}\label{lemma_compl_of_Z_p-mod} Let $\mathcal M$ be a $\ZZ_p[C]$-module, which is finitely generated over $\ZZ_p$. Then $\mathcal M_I=\mathcal M_{I_p}$ the homomorphism 
$ \mathcal M \to \hat{\mathcal M}_I$ splits in the category of $C$-modules and there is a natural isomorphism 
$$\mathcal M=\mathcal MI^\infty \oplus \hat{\mathcal M}_I,$$
where $\mathcal MI^\infty=\bigcap \mathcal MI^i.$
Moreover, if $C=\langle t \rangle$ and $\mathcal M(t-1)\subseteq p\mathcal M,$ then $\mathcal M\cong \hat{\mathcal M}_I.$
\end{Lemma}
\begin{proof} Since $\mathcal M$ is a finitely generated $\ZZ_p$-module, $\mathcal M$ is isomorphic to $\ZZ_p^n\oplus \bigoplus_{j=1}^m \ZZ/p^{k_j}$ as an abelian group.
The quotient $\mathcal M/p^i$ is a $\ZZ/p^i[C]$-module, which is finitely generated over $\ZZ/p^i.$  It follows that the map $\mathcal M/p^i \to (\mathcal M/p^i)^\wedge_I$ is a split epimorphism of $\ZZ/p^i[C]$-modules. Denote the section by $s_i:(\mathcal M/p^i)^\wedge_I \to  \mathcal M/p^i.$ Moreover, using that the section is natural, we obtain that the diagram
$$
\begin{tikzcd}
(\mathcal M/p^{i})^\wedge_I\arrow[r,"s_i"]\arrow[d] &\mathcal M/p^{i}  \arrow[d] \\
(\mathcal M/p^{i-1})^\wedge_I\arrow[r,"s_{i-1}"] & \mathcal M/p^{i-1}\\
\end{tikzcd}
$$
is commutative. It follows that $\{s_i\}$ induce a section on the level of inverse limits. Using Corollary \ref{cor_comm_Z_p-mod}, we obtain that $\mathcal M \to \hat{\mathcal M}_I$ is a split epimorphism of $C$-modules and the section is natural. The kernel of the epimorphism equals to $\mathcal MI^\infty.$ Then $\mathcal M=\mathcal MI^\infty \oplus \hat{\mathcal M}_I.$ If $\mathcal MI=\mathcal M(t-1)\subseteq p\mathcal M,$ then $\mathcal MI^i\subseteq p^i\mathcal M,$ and hence, $\mathcal MI^\infty=0$ and $\mathcal M=\hat{\mathcal M}_I.$
\end{proof}

\subsection{Completion of metabelian groups over $C$}

\begin{Lemma}\label{lemma_homology_of_nilpotent} Let $G$ be a group,  $N$ be a nilpotent $G$-module and $X$ be an abelian group. Then $H_n(N,X)$ is a nilpotent $G$-module.
\end{Lemma}
\begin{proof} Let $N$ be a nilpotent $G$-module of class $m.$
The proof is by induction on $m.$ If $m=1,$ then $N$ is a trivial $G$-module, and hence, $H_n(N,X)$ is trivial. Prove the inductive step. Consider a short exact sequence $T\mono N\epi N',$ where $T$ is a trivial module and $N'$ is a nilpotent module of class $m-1.$ By inductive hypothesis the second page of the spectral sequence consists of nilpotent modules $H_i(N',H_j(T,X)).$ It follows that $H_n(N,X)$ is nilpotent.  
\end{proof}

By $C$ we denote the infinite cyclic group.  
A group $G$ is said to be metabelian over $C$ if there is a short exact sequence $M\mono G \epi C,$ where $M$ is abelian. Then $M$ can be considered as a $C$-module, where the action of $C$ is induced by the conjugation. It is easy to see that $G=M\rtimes C.$ Hence the group structure on $G$ depends only on the $C$-module $M.$ Then the completions can be described as follows  
\begin{equation}
\hat G_\ZZ=\hat M_I \rtimes C, \hspace{1cm} \hat G_{\ZZ/p}=\hat M_{I_p}\rtimes \ZZ_p, \hspace{1cm} \hat G_\QQ=(M\otimes \QQ)^\wedge_I\rtimes \QQ,
\end{equation}
where $I$ is the augmentation ideal of $\ZZ[C]$ and $I_p=I+p\cdot \ZZ[C]$ (see Prop. 4.7 and Prop. 4.12 of \cite{Ivanov_Mikhailov}).

\begin{Proposition}\label{proposition_MrtimesC} Let $M$ be a $C$-module such that $M\otimes \QQ$ is finite dimensional and $G=M\rtimes C$. Then the map $$H_n(G,\QQ)\epi H_n(\hat G_\QQ,\QQ) $$ is an epimorphism for any $n.$ 
\end{Proposition}
\begin{proof} Set $V=M\otimes \QQ.$
The homomorphism $G\to \hat G_\QQ$ induces a morphism of the spectral sequences $$H_i(C,H_j(M,\QQ))\longrightarrow H_i(\QQ,H_j(\hat V_I,\QQ)).$$ Note that $H_i(C,-)=0$ and $H_i(\mathbb Q,-)=0$ for $i\geq 2.$ Indeed, for $C$ it is obvious and for $\QQ$ it follows from the fact that homology commutes with direct limits ($H_*(\varinjlim G_i, M)=\varinjlim H_*(G_i,M)$) and $\QQ=\varinjlim \ZZ$. Then the spectral sequences have only two nontrivial columns and the morphism of spectral sequences gives a morphism of the short exact sequences
$$
\begin{tikzcd} 0 \arrow[r] & H_0(C, H_n(M,\QQ)) \arrow[r] \arrow[d] & H_n(G,\QQ) \arrow[r]\arrow[d] & H_1(C,H_{n-1}(M,\QQ)) \arrow[r]\arrow[d] & 0\\
0 \arrow[r] & H_0(\QQ, H_n(\hat V_I,\QQ)) \arrow[r] & H_n(\hat G_\QQ,\QQ) \arrow[r] & H_1(\QQ,H_{n-1}(\hat V_I,\QQ)) \arrow[r] & 0.
\end{tikzcd}
$$ Therefore, it is enough to prove that
$$H_i(C,H_n(M,\QQ)) \longrightarrow H_i(\QQ,H_n(\hat V_I,\QQ))$$
is an epimorphism for any $n$ and $i=0,1.$ 

Note that $H_*(C,\QQ)=H_*(\QQ,\QQ).$ Using this and the long exact sequence associated with a short exact sequence of modules, we obtain 
$$H_*(C,\mathcal N)=H_*(\QQ,\mathcal N)$$ for any nilpotent $\QQ[\QQ]$-module $\mathcal N.$ If $\mathcal N$ is a nilpotent $\QQ[C]$-module, there is a unique way to lift the action of $C=\langle t\rangle$ to the action of $\QQ=\{t^\alpha\mid \alpha\in \QQ\}$ on $\mathcal N:$
$$xt^{\alpha}=\sum_{n=0}^{\infty} \binom{\alpha}{n} x(t-1)^n,$$
where $\alpha\in \QQ,$ $x\in \mathcal N$ (see Lemma 4.4 of \cite{Ivanov_Mikhailov}).  The sum is finite because $\mathcal N(t-1)^n=0$ for big enough $n$. It is easy to see from the formula that $\mathcal N(t^\alpha-1)\subseteq \mathcal N(t-1).$ It follows that the $\mathbb Q[\QQ]$-module $\mathcal N$ is nilpotent if and only if it is nilpotent as a $\QQ[C]$-module. 

 The module $\hat V_I$ is nilpotent as a $ \QQ[C]$-module (Lemma \ref{lemma_k[c]-mod}). Then $H_{n}(\hat V_I,\QQ)$ is nilpotent as $\QQ[C]$-module (Lemma \ref{lemma_homology_of_nilpotent}), and hence, it is nilpotent as a $\QQ[\QQ]$-module. It follows that $$H_i(C,H_{n}(\hat V_I,\QQ))=H_i(\QQ,H_{n}(\hat V_I,\QQ)).$$ Hence it is sufficient to prove that the homomorphism
$$H_n(M,\QQ)\longrightarrow H_n(\hat V_I,\QQ)$$
is a split epimorphism in the category of $C$-modules.  Note that $H_n(M,\QQ)= \Lambda^n (V)$ and $H_n(\hat V_I,\QQ)=\Lambda^n(\hat V_I).$ Then the assertion follows from the fact that $V\to \hat V_I$ is a split epimorphism of $C$-modules (Lemma \ref{lemma_k[c]-mod}).  
\end{proof}

\begin{Theorem}\label{theorem_MrtimesC} Let  $M$ be a finitely generated $C$-module such that ${\sf tor}(M)$ is finite and $M\otimes \QQ $ is finite dimensional. Set $G=M\rtimes C.$ Then $$H_*(\hat G_\ZZ,\ZZ/p)=H_*(\hat G_{\ZZ/p},\ZZ/p)$$ and the morphism 
$$H_n(G,\ZZ/p) \longrightarrow H_n(\hat G_{\ZZ/p},\ZZ/p)$$
is an epimorphism for any $n$. Moreover, if $C=\langle t \rangle$ and $M(t-1)\subseteq pM,$ then 
$$H_*(G,\ZZ/p)\cong H_*(\hat G_{\ZZ/p},\ZZ/p) \hspace{1cm}\text{and}\hspace{1cm} (BG)_{\ZZ/p}\cong B(\hat G_{\ZZ/p}).$$
\end{Theorem}
\begin{proof}
Set $\mathcal M=\hat M_{\ZZ/p}.$ Then $\mathcal M$ is a finitely generated $\ZZ_p$-module (Proposition \ref{prop_pro-p-abelian}) and it is isomorphic to $\ZZ_p^m \oplus \bigoplus_{j=1}^m \ZZ/p^{k_j}$ as an abelian group. Using Corollary \ref{cor_comm_compl} and Lemma \ref{lemma_compl_of_Z_p-mod}  we obtain $$\hat{\mathcal M}_I=\hat{\mathcal M}_{I_p}=(\hat M_I)^{\wedge}_{\ZZ/p}=\hat M_{I_p}, \hspace{1cm} \mathcal M = \mathcal MI^\infty \oplus \hat{\mathcal M}_I.$$
It follows that 
\begin{equation}\label{eq_last_th_proof1}
H_*(\mathcal M,\ZZ/p) \longrightarrow H_*(\hat{\mathcal M}_I,\ZZ/p)
\end{equation}
is a split epimorphism of $C$-modules. By Proposition  \ref{prop_pro-p-abelian} we obtain $H_*(\mathcal M,\ZZ/p)=H_*(M,\ZZ/p)$ and $H_*(\hat M_{I_p},\ZZ/p)=H_*(\hat{\mathcal M}_I,\ZZ/p)=H_*(\hat M_I,\ZZ/p).$ It follows that 
\begin{equation}\label{eq_last_th_proof2}
H_*(M,\ZZ/p) \longrightarrow H_*(\hat M_I,\ZZ/p)
\end{equation}
is a split epimorphism of $C$-modules and 
$$H_*(\hat M_I,\ZZ/p)=H_*(\hat M_{I_p},\ZZ/p).$$ 
By Theorem \ref{theorem_finite_prufer_solvable} we get that $H_n(\hat M_{I_p},\ZZ/p)$ is a finite $\ZZ/p[\ZZ_p]$-module. Using Proposition \ref{prop_pro-p-abelian}, we obtain 
$$H_*(C,H_n(\hat M_I,\ZZ/p))=H_*(C,H_n(\hat M_{I_p},\ZZ/p))=H_*(\ZZ_p,H_n(\hat M_{I_p},\ZZ/p)).$$
Since $\hat G_\ZZ=\hat M_I\rtimes C$ and $\hat G_{\ZZ/p}=\hat M_{I_p}\rtimes \ZZ_p,$ it follows that $$H_*(\hat G_{\ZZ},\ZZ/p)=H_*(\hat G_{\ZZ/p},\ZZ/p).$$
The homomorphism $G\to \hat G_\ZZ$ induces a morphism of the spectral sequences $$H_i(C,H_j(M,\ZZ/p))\longrightarrow H_i(C,H_j(\hat M_I,\ZZ/p)).$$ 
The spectral sequences have only two nontrivial columns and the morphism of spectral sequences gives a morphism of the short exact sequences
\begin{equation}\label{eq_last_th_proof3}
\begin{tikzcd} 0 \arrow[r] & H_0(C, H_n(M,\ZZ/p)) \arrow[r] \arrow[d] & H_n(G,\ZZ/p) \arrow[r]\arrow[d] & H_1(C,H_{n-1}(M,\ZZ/p)) \arrow[r]\arrow[d] & 0\\
0 \arrow[r] & H_0(C, H_n(\hat M_I,\ZZ/p)) \arrow[r] & H_n(\hat G_\ZZ,\ZZ/p) \arrow[r] & H_1(C,H_{n-1}(\hat M_I,\ZZ/p)) \arrow[r] & 0.
\end{tikzcd}
\end{equation}
Using that the left and the right vertical arrows are split epimorphisms, we obtain that the middle vertical arrow is an epimorphism. 

Assume now $M(t-1)\subseteq pM.$ Then $\mathcal M(t-1)\subseteq p\mathcal M.$ Hence, $\mathcal MI^{\infty}=0.$ It follows that the map \eqref{eq_last_th_proof1} is an isomorphism. Then \eqref{eq_last_th_proof2} is an isomorphism. Then the left and the right vertical arrows in the diagram \eqref{eq_last_th_proof3} are isomorphisms. Then $H_*(G,\ZZ/p)= H_*(\hat G_{\ZZ},\ZZ/p)=H_*(\hat G_{\ZZ/p},\ZZ/p).$ Then $BG \to B(\hat G_{\ZZ/p})$ is a $\ZZ/p$-homological equivalence. Moreover, the space $B(\hat G_{\ZZ/p})$ is $\ZZ/p$-local because the group $\hat G_{ZZ/p}$ is $H\ZZ/p$-local. It follows that $(BG)_{\ZZ/p}=B(\hat G_{\ZZ/p}).$
\end{proof}

\begin{Corollary} Let $M$ be a $C$-module such that $M\rtimes C$ is finitely presented. Set $G=M\rtimes C.$ Then $$H_*(\hat G_\ZZ,\ZZ/p)=H_*(\hat G_{\ZZ/p},\ZZ/p)$$ and the morphism 
$$H_n(G,\ZZ/p) \longrightarrow H_n(\hat G_{\ZZ/p},\ZZ/p)$$
is an epimorphism for any $n$. Moreover, if $C=\langle t \rangle$ and $M(t-1)\subseteq pM,$ then 
$$H_*(G,\ZZ/p)\cong H_*(\hat G_{\ZZ/p},\ZZ/p) \hspace{1cm}\text{and}\hspace{1cm} (BG)_{\ZZ/p}\cong B(\hat G_{\ZZ/p}).$$
\end{Corollary}

\begin{Corollary} Let $a\in {\rm GL}_n(\ZZ)$ such that entries of $a-1$ are divisible by $p.$ Set $G=\ZZ^n\rtimes_a C.$ Then $(BG)_{\ZZ/p}=B(\hat G_{\ZZ/p}).$
\end{Corollary}

\end{document}